\pdfoutput=1
\RequirePackage{ifpdf}
\ifpdf 
\documentclass[pdftex]{sigma}
\else
\documentclass{sigma}
\fi

\numberwithin{equation}{section}

\newtheorem{Theorem}{Theorem}[section]
\newtheorem*{Theorem*}{Theorem}
\newtheorem{Corollary}[Theorem]{Corollary}
\newtheorem{Lemma}[Theorem]{Lemma}
\newtheorem{Proposition}[Theorem]{Proposition}
\newtheorem{Assumption}[Theorem]{Assumption}

\theoremstyle{definition}
\newtheorem{Definition}[Theorem]{Definition}

\newtheorem{Remark}[Theorem]{Remark}

 \def\L{{\mathcal L}}
 \def\P{{\mathcal P}}

 \def\D{{\mathcal D}}

 \def\p{{\mathbb P}}
 \def\e{{\mathbb E}}
 \def\r{{\mathbb R}}
 \def\wr{{\textnormal{Wr}}}
 
 \def\d{{\textnormal d}}
 \def\i{{\textnormal i}}

\begin{document}

\allowdisplaybreaks

\newcommand{\arXivNumber}{2405.11051}

\renewcommand{\PaperNumber}{099}

\FirstPageHeading

\ShortArticleName{Darboux Transformation of Diffusion Processes}

\ArticleName{Darboux Transformation of Diffusion Processes}

\Author{Alexey KUZNETSOV and Minjian YUAN}
\AuthorNameForHeading{A.~Kuznetsov and M.~Yuan}
\Address{Department of Mathematics and Statistics, York University, \\ 4700 Keele Street, Toronto, ON, M3J~1P3, Canada}
\Email{\mail{akuznets@yorku.ca}, \mail{yuanm@yorku.ca}}
\URLaddress{\url{https://kuznetsovmath.ca/}}

\ArticleDates{Received February 04, 2025, in final form November 12, 2025; Published online November 24, 2025}

\Abstract{Darboux transformation of a second-order linear differential operator is a well-known technique with many applications in mathematics and physics. We study Darboux transformation from the point of view of Markov semigroups of diffusion processes. We construct the Darboux transform of a diffusion process through a combination of Doob's $h$-transform and a version of Siegmund duality. Our main result is a simple formula that connects transition probability densities of the two processes. We provide several examples of Darboux transformed diffusion processes related to Brownian motion and Ornstein--Uhlenbeck process. For these examples, we compute explicitly the transition probability density and derive its spectral representation. }


\Classification{60J60; 60J35}

\section{Introduction}\label{section_intro}

Darboux transformation is a well-known technique for studying second-order linear differential operators, dating back to the late 19th century \cite{Darboux_1882}.
To introduce this transformation, consider a second-order linear differential operator
 \begin{equation}\label{def_L}
 {\mathcal L}=\frac{1}{2} \partial_y^2 - c(y),
 \end{equation}
 where $y \in (l,r)\subseteq \r$. We assume that $c(y)$ is continuous on $(l,r)$ and that $\L$ acts on $C^2$ functions on $(l,r)$. Let $h$ be a positive function on $(l,r)$ that satisfies $\L h = \lambda h$ for some $\lambda \in \r$. It is straightforward to check that the operator $\L-\lambda$ can be factorized as a product of two first-order differential operators
 \begin{equation}\label{L_factorization}
 \L-\lambda =\frac{1}{2}\D_{\frac{1}{h}} \D_h,
 \end{equation}
 where we denoted
\[
{\mathcal D}_{h}:=\partial_y-\frac{h'(y)}{h(y)}=h(y) \partial_y \frac{1}{h(y)}, \qquad
{\mathcal D}_{\frac{1}{h}}:=\partial_y+\frac{h'(y)}{h(y)}=\frac{1}{h(y)} \partial_y h(y).
\]

We now introduce the Darboux transform of $\L$ as a second-order linear differential operator~$\widetilde \L$ defined by
 \begin{equation}\label{tilde_L_factorization}
 \widetilde \L-\lambda =\frac{1}{2}\D_h \D_{\frac{1}{h}}.
 \end{equation}
Thus, the Darboux transform of a second-order linear operator is obtained by factorizing it (or more precisely, its shift by $\lambda$) as a product of two first-order differential operators and then interchanging the order of the factors. The function $h$ is called the {\it seed function} of the Darboux transformation. It is an easy exercise to verify that $\widetilde \L$ is given by
\begin{equation}\label{def_tilde_L}
\widetilde \L = \frac{1}{2} \partial_y^2 - \tilde c(y),
\end{equation}
where
\[
\tilde c(y):=c(y) - \partial_y^2 \ln h(y)=\left(\frac{h'(y)}{h(y)}\right)^2-c(y)-2\lambda.
\]
A key property of the Darboux transformation, which makes it so useful in applications, is the following intertwining relation
\begin{equation}\label{L_tilde_L_intertwining}
 \widetilde \L \D_h = \D_h \L,
\end{equation}
which follows directly from \eqref{L_factorization} and \eqref{tilde_L_factorization}. In particular, if $f$ solves $\L f=\mu f$ then
$g=\D_h f$ solves $\widetilde \L g=\mu g$. Moreover, since $\widetilde \L (1/h)=\lambda (1/h)$ (a consequence of
\eqref{tilde_L_factorization}), the Darboux transformation is invertible: starting from $\widetilde \L$ and applying the Darboux transformation with the seed function $\tilde h(y)=1/h(y)$ recovers the original operator $\L$. The Darboux transformation can also be iterated, leading to the so-called Darboux--Crum transformations (see \cite{Crum_1955} and \cite[Section 2.4]{Gomez_2020}).

There exists a vast literature on applications of the Darboux transformation in applied mathematics and physics. For example, it is used to construct exactly solvable potentials for Schr\"odinger equation \cite{Bagrov_1997} and it plays a crucial role in the theory of exceptional orthogonal polynomials \cite{Garcia_2019, Ullate_2014}. Many additional applications in integrable systems are described in~\cite{Gu_2005}
and~\cite{Matveev_1991}. However, to the best of our knowledge, the Darboux transformation has not yet been studied in the context of diffusion processes, which is the focus of the present paper.

Assuming that the functions $c(y)$ and $\tilde c(y)$ in \eqref{def_L} and \eqref{def_tilde_L} are nonnegative, we can regard the operators $\L$ and $\widetilde \L$ as the infinitesimal generators of killed Brownian motion processes $Y$ and~$\widetilde Y$ with state space $(l,r)$. When a boundary point $l$ or $r$ is non-singular for one of these processes (see Section~\ref{section_Siegmund} for a review of Feller's boundary classification), there are many different ways to construct a diffusion process with a given infinitesimal generator by specifying different boundary conditions (reflecting/killing/elastic/etc.). Instead of first fixing the boundary conditions and then studying the resulting processes, we will proceed in the opposite direction: we begin by setting a goal for the type of process we wish to construct, and then determine the boundary conditions and other ingredients needed for this construction to succeed.

What follows is an informal description of what we plan to do. Let $Y$ be a killed Brownian motion on $(l,r)$ with the Markov semigroup $\{\P_t\}_{t\ge 0}$ (where we denoted $\P_t f(y):=\e_y [f(Y_t)]$) and
with the infinitesimal generator $\L$ given by \eqref{def_L}. Suppose we have another killed Brownian motion process
$\widetilde Y$ with the Markov semigroup \smash{$\big\{\widetilde \P_t\big\}_{t\ge 0}$}
and the infinitesimal generator $\widetilde \L$ as in~\eqref{def_tilde_L}. Moreover, assume that the intertwining relationship for generators \eqref{L_tilde_L_intertwining} extends to the semigroups, so that $\widetilde \P_t \D_h = \D_h \P_t$. Our aim is to use this intertwining relationship to express the transition operators $\widetilde \P_t$ in terms of $\P_t$. One way to define operators $\widetilde \P_t$ which satisfy such intertwining relationship is via
\begin{equation}\label{def_tildeP}
 \widetilde \P_t=\D_h \P_t {\mathcal I}^{(z)}_{h},
\end{equation}
where $z$ is some fixed point in $(l,r)$ and the integral operator ${\mathcal I}^{(z)}_{h}$ is defined by
\[
{\mathcal I}^{(z)}_{h} f(y):=h(y) \int_z^y \frac{f(u)}{h(u)} \d u.
\]
 Assume that $h$ is a positive $\lambda$-invariant function for $\P_t$, meaning that it satisfies $\P_t h={\rm e}^{\lambda t} h$. This condition implies $\L h=\lambda h$. Then the operators $\widetilde \P_t$ defined in \eqref{def_tildeP} would indeed satisfy the intertwining relationship $\widetilde \P_t \D_h = \D_h \P_t$, since \smash{${\mathcal I}^{(z)}_{h} \D_h f(y)=f(y)-f(z)h(y)/h(z)$} for $C^1$ functions $f$ and $\D_h \P_t h=0$ (due to $\lambda$-invariance of $h$). A similar argument shows that the operators defined in \eqref{def_tildeP} would not depend on the choice of $z \in (l,r)$. If the semigroups
of $Y$ and~$\widetilde Y$ do satisfy \eqref{def_tildeP}, then it is easy to verify that the transition probability densities of $\widetilde Y$ and~$Y$ should be related by
\begin{equation}\label{equation_tilde_p0}
 p_t^{\widetilde Y}(x,y) \stackrel{?}{=}\frac{h(x)}{h(y)} \partial_{x}
\left[ \frac{1}{h(x)} \int_{y}^r p_t^Y(x,u) h(u) \d u \right].
\end{equation}

We add a question mark to emphasize that the above identity is, at this stage, only a hypothesis. Indeed, the discussion above is non-rigorous, and this construction of $\widetilde Y$ contains several gaps. For instance, it is not clear why the operators
\eqref{def_tildeP} would indeed define a Markov semigroup, in particular, why the right-hand side of \eqref{equation_tilde_p0} should be positive. It is also unclear whether specific boundary conditions on $Y$ and~$\widetilde Y$ are needed for \eqref{equation_tilde_p0} to hold. Nevertheless, this informal reasoning is essentially correct, as we will show in this paper. Our main result is the construction of a killed Brownian motion $\widetilde Y$ such that the transition probabilities of $Y$ and~$\widetilde Y$ satisfy an analogue of~\eqref{equation_tilde_p0} (the only modification being the addition of a constant to~$\tilde c(y)$ to ensure positivity, which introduces an extra exponential factor in \eqref{equation_tilde_p0}). We will demonstrate that~$\widetilde Y$ can be constructed in three steps, starting from a killed Brownian motion $Y$ and a~$\lambda$-invariant function $h$: apply Doob’s $h$-transform, then a version of Siegmund duality \cite{Siegmund_1976}, followed by another Doob’s $h$-transform.

The primary motivation for constructing Darboux-transformed diffusion processes is that this would greatly expand the number of diffusions for which the transition probability density can be computed in closed form. There are only a few examples (Brownian motion, Ornstein–Uhlenbeck process, Bessel process, square-root diffusion process, and several others) for which the transition probability is known in closed form, and all of these processes are very useful in applications, particularly in mathematical finance. Expanding the number of such examples is, therefore, a~worthwhile task.

The paper is organized as follows. In Section~\ref{section_preliminaries}, we recall some basic facts about diffusion processes, including Feller’s classification of boundary points for diffusion processes. In Section~\ref{section_Siegmund}, we discuss Siegmund duality \cite{Clifford_1985,Siegmund_1976} and present a new version of it (which can be applied to processes with reflecting instead of absorbing boundary conditions). Our main result in this section is Theorem~\ref{thm:main}, which establishes a connection between transition semigroups of the process and its Siegmund dual (in our new construction). This result could be of independent interest, as Siegmund duality is an important tool in the study of interacting particle systems~\mbox{\cite{Assiotis2018,Assiotis2023,Assiotis2019}}. In~Section~\ref{section_Darboux_diffusion}, we construct the Darboux transform of a killed Brownian motion and in our second main result, Theorem~\ref{thm:main2}, we show that identity \eqref{equation_tilde_p0} holds (up to an exponential factor). We~also present an example that illustrates the necessity of $\lambda$-invariance of $h$: we show that if $h$ solves $\L h=\lambda h$ but is not $\lambda$-invariant for the process $Y$, then the right-hand side of~\eqref{equation_tilde_p0} may be negative. Section~\ref{section_Examples} provides five examples of the Darboux transformation applied to various diffusion processes. In the first four examples, we start with the Brownian motion process on an interval with various boundary conditions, and in the last example we apply the Darboux transformation to the Brownian motion killed at rate $y^2/2$ (this latter process is closely related to the Ornstein–Uhlenbeck process). In all examples, we describe the boundary behavior and derive explicit formulas for the transition probability density of the transformed process. In four of these examples, we also provide the spectral representation of the transition probability density. At the end of Section~\ref{section_Examples}, we discuss connections between our results and those obtained in the physics literature on propagators for the one-dimensional Schr\"odinger equation.

\section{Preliminaries}\label{section_preliminaries}

In this section, we introduce notation and recall several basic facts about diffusion processes, which will be useful later. We will mostly follow \cite[Chapter II]{Borodin2002}; other good references are~\cite{ItoMcKean_1974,Mandl1968}.
Let $X$ be a regular diffusion process on an interval $(l,r)\subseteq \r$. We assume that the infinitesimal generator of $X$ has the form
$
{\mathcal L}_X=\frac{1}{2} \sigma^2(x) \partial_x^2 + b(x) \partial_x - c(x)$,
where the functions $b$, $c$ and $\sigma$ are continuous on $(l,r)$ and $\sigma(x)>0$ and $c(x)\ge 0$ for $x\in (l,r)$.
The operator ${\mathcal L}_X$ is acting on the class of $C^2$ functions on $(l, r)$ with appropriate boundary conditions, which will be described below. We denote by $s(x)$ the scale function of $X$ and by~$\d m(x)$ and $\d k(x)$ the speed measure and the killing measure of $X$. The derivative of $s$ and the densities of measures $\d m(x)$ and $\d k(x)$ can be found via
\begin{equation}\label{eqn:s_m_k_prime}
s'(x)={\rm e}^{-B(x)}, \qquad m'(x)=2\sigma^{-2}(x) {\rm e}^{B(x)}, \qquad k'(x)=c(x) m'(x),
\end{equation}
where $B'(x):=2 \sigma^{-2}(x) b(x)$ (see \cite[Section II.9]{Borodin2002}).
We also denote
\[
f^{\pm}(x)=\lim\limits_{h\to 0^{\pm}} \frac{f(x+h)-f(x)}{s(x+h)-s(x)}.
\]

Let us now recall Feller's boundary classification (see \cite[Section II.6]{Borodin2002}). Fix a point $z\in (l,r)$ and set
\begin{gather}\label{def:R_Q}
R(x):=(m(z)-m(x)+k(z)-k(x)) s'(x), \!\qquad Q(x):=(s(z)-s(x)) (m'(x)+k'(x)).\!\!\!
\end{gather}
The left boundary $l$ is called an {\it exit} boundary if $R\in L_1((l,z))$ and an {\it entrance} boundary if~${Q \in L_1((l,z))}$. \label{R_Q_L1} The conditions for the right boundary $r$ are the same, with the obvious change of the interval $(l,z) \mapsto (z,r)$. A boundary point that is neither entrance nor exit is called {\it natural}, while a boundary point that is both entrance and exit is called {\it non-singular}. If $l$ (or $r$) is a~non-singular boundary point for $X$, we impose one of the following boundary conditions:
\label{elastic_boundary_condition}
\begin{itemize}\itemsep=0pt
\item[(i)] {\it elastic} boundary condition: $f(l+)=\alpha f^+(l+)$ for $\alpha >0 $ (or $f(r-)=-\beta f^-(r-)$ for $\beta > 0$);
\item[(ii)] {\it reflecting} boundary condition: $f^+(l+)=0$ (or $f^-(r-)=0$);
\item[(iii)] {\it killing} boundary condition: $f(l+)=0$ (or $f(r-)=0$).
\end{itemize}
In this paper, we do not consider sticky boundary conditions or traps (in the terminology of \cite[Section II.7]{Borodin2002}) -- these involve the second derivative at a point and result in a diffusion process whose distribution has an atom at the boundary.

Let $\lambda>0$ and denote by $\psi_{\lambda}(x)$ and $\varphi_{\lambda}(x)$ the {\it fundamental solutions} to the second-order ODE $\L_X f=\lambda f$. We recall that the fundamental solutions are unique (up to multiplication by a constant) positive solutions to $\L_X f=\lambda f$ such that $\psi_{\lambda}$ is increasing and $\varphi_{\lambda}$ is decreasing
(see~\cite[Section II.10]{Borodin2002}). Also, if $l$ is a non-singular boundary point, then $\psi_{\lambda}$ must satisfy one of the boundary conditions stated above (and the same applies for $\varphi_{\lambda}$ at the right boundary point~$r$). Their Wronskian is defined as
 \begin{equation*}
\wr[\varphi_{\lambda}; \psi_{\lambda}]:=\varphi_{\lambda}(x) \psi'_{\lambda}(x)- \varphi'_{\lambda}(x)\psi_{\lambda}(x),
 \end{equation*}
 and it is known that the quantity $\omega_{\lambda}:=\wr[\varphi_{\lambda}; \psi_{\lambda}]/s'(x)$ is independent of $x$. We denote by~$q^X_t(x,y)$ the transition probability density of the process $X$ with respect to the speed measure, that is
 \[
 \p_x(X_t \in A)=\int_A q^X_t(x,y) \d m(y),
 \]
 for all Borel sets $ A \subset (l,r)$.
 The function $q^X_t(x,y)$ is symmetric in $x$ and $y$ (see \cite[Section~II.4]{Borodin2002}). The Green's function of the process $X$ is given in terms of the fundamental solutions $\psi_{\lambda}$ and~$\varphi_{\lambda}$~by
\begin{gather}\label{Green_X}
 G^X_{\lambda}(x,y):=
 \int_0^{\infty} {\rm e}^{-\lambda t} q^X_t(x,y) \d t=
 \begin{cases}
 \omega_{\lambda}^{-1} \psi_{\lambda}(x) \varphi_{\lambda}(y) & \text{if } x\le y, \\
 \omega_{\lambda}^{-1} \psi_{\lambda}(y) \varphi_{\lambda}(x)  & \text{if } y\le x,
 \end{cases}
\end{gather}
as stated in \cite[Section II.11]{Borodin2002}. The transition probability density with respect to Lebesgue measure will be denoted by
\smash{$
p_t^X(x,y):=q^X_t(x,y) m'(y)$}.

The above formula \eqref{Green_X} for the Green's function provides one possible way of finding the transition probability density of a diffusion process $X$ with specified boundary behaviour: the Green's function is the Laplace transform in the $t$-variable of $q_{t}^X(x,y)$ and thus uniquely determines the transition probability density. Conversely, given the transition probability density $q_t^{X}(x,y)$ of a~diffusion process, we can determine the boundary behaviour of $X$ at each non-singular boundary by first finding the Green's function, then determining the fundamental solutions $\psi_{\lambda}$ and $\varphi_{\lambda}$ from~\eqref{Green_X}, and finally studying the boundary behaviour of the fundamental solutions. For example, if $l$ is a non-singular boundary and we find that $\psi_{\lambda}^+(l+)=0$ ($\psi_{\lambda}(l+)=0$) for all~${\lambda>0}$, then $X$ has a reflecting (respectively, killing) boundary condition at $l$.

We also note the following fact: if a process $X=\{X_t\}_{t\ge 0}$ is mapped into a process $Y=\{y(X_t)\}_{t\ge 0}$ via an increasing differentiable function $y=y(x)$, then the scale function and speed measure are transformed as follows
$
s_Y(y(x))=s_X(x)$, $ m_Y'(y(x))=m_X'(x)/y'(x)$.
The state space for the process $Y$ is the interval $(y(l), y(r))$, and the boundary behavior of $Y$ at points~$y(l)$ and $y(r)$ is the same as the boundary behavior of $X$ at $l$ and $r$.

 \section{Siegmund duality}\label{section_Siegmund}

Let $Z$ and $\widetilde Z$ be two Markov processes on $[0,\infty)$. We say that $\widetilde Z$ is a Siegmund dual of $Z$ if
\begin{equation}\label{Siegmund1}
\p_y(Z_t \ge x)=\p_x\bigl(\widetilde Z_t\le y\bigr),
\end{equation}
for all $x,y,t \ge 0$. This concept was introduced by Siegmund \cite{Siegmund_1976} in 1976 and was later studied in \cite{Clifford_1985,Cox_1984,Kolokoltsov_2011}. Diffusion processes satisfying \eqref{Siegmund1} were called conjugate diffusions in \cite{Toth_1996}. The existence of a dual process $\widetilde Z$ requires the process $Z$ to be {\it stochastically monotone}, see \cite{Clifford_1985,Siegmund_1976}. The latter property can be stated as follows: for any $t> 0$ and $x>0$ the function $y\mapsto \p_y(Z_t \ge x)$ is non-decreasing.

Siegmund duality is applied when one of the processes has an absorbing boundary. For example, assuming that $Z$ is conservative and taking the limit of~\eqref{Siegmund1} as $x \to 0^+$, we see that the process $\widetilde Z$ must have an absorbing boundary at zero. Thus, in this case, the distribution of~$\widetilde Z_t$ will have an atom at zero, and assuming that this distribution has a density on $(0,\infty)$, it must be given by
\smash{$p_t^{\widetilde Z}(x,y)=\partial_y \p_y(Z_t \ge x)$},
which follows from \eqref{Siegmund1} by differentiation with respect to $y$.

 The following theorem is our main result in this section: we construct a version of Siegmund duality that applies to processes with reflecting or killing boundary conditions.

\begin{Theorem}\label{thm:main}${}$
Assume that $(l,r) \subseteq \r$, $b(x) \in C((l,r))$, $\sigma(x)\in C^1((l,r))$ and $\sigma(x)>0$ for~${x \in (l,r)}$. Let $X$ be a conservative diffusion process on $(l,r)$ with infinitesimal generator
 \[
 \L_X=\frac{1}{2} \sigma^2(x) \partial_x^2+b(x) \partial_x,
 \]
 and reflecting boundary conditions at every non-singular boundary point. Then there exists a~diffusion process $\widetilde X$ such that
\begin{equation}\label{eqn:main}
\p_{x_1}(X_t\le y)=\p_{x_2}(X_t\le y)+\p_y\bigl(x_1<\widetilde X_t \le x_2\bigr),
\end{equation}
for all $t>0$ and $x_i,y \in (l,r)$ for which $x_1<x_2$. The process $\widetilde X$ has infinitesimal generator
\begin{equation}\label{def_tilde_L2}
\L_{\widetilde X}=\frac{1}{2} \sigma^2(x) \partial_x^2+(\sigma(x)\sigma'(x) - b(x) )\partial_x
\end{equation}
and killing boundary condition at every non-singular boundary point.
\end{Theorem}

\begin{proof}
We denote $\tilde b(x):=\sigma(x)\sigma'(x) - b(x)$ and define the process $\widetilde X$ as the solution to the stochastic differential equation (SDE)
$
\d \widetilde X_t=\tilde b\bigl(\widetilde X_t\bigr) \d t+\sigma\bigl(\widetilde X_t\bigr) \d W_t$,
which is killed at the first exit from $(l,r)$. The existence of a unique in law weak solution of this SDE (up to the first exit from $(l,r)$) is guaranteed by \cite[Theorem 5.15]{Karatzas_Shreve_1988}.

Let $s(x)$ and $m(x)$ denote the scale function and the speed measure of $X$. Using \eqref{eqn:s_m_k_prime}, we check that
 \begin{equation}\label{L_tilde_L_speed_scale}
 \L_X=\frac{\d}{\d m(x)}\frac{\d}{\d s(x)}, \qquad
 \L_{\widetilde X}=\frac{\d}{\d s(x)}\frac{\d}{\d m(x)},
 \end{equation}
thus the speed measure/scale function of $\widetilde X$ is equal to the scale function/speed measure of $X$ (see \cite[Section 4]{Cox_1984} for a similar construction in the case of classical Siegmund duality).

The relation \eqref{L_tilde_L_speed_scale} imposes restrictions on the boundary behavior of $X$ and $\widetilde X$. The type of the boundary behavior depends on the functions $R$ and $Q$ defined via \eqref{def:R_Q}, and when we compute their analogues $\widetilde R$ and $\widetilde Q$ for the process $\widetilde X$ we find that
 $\widetilde Q\equiv R$ and $\widetilde R \equiv Q$ (because the killing measure is zero and
 $m_X(s)=s_{\widetilde X}(s)$ and $s_X(x)=m_{\widetilde X}(x)$).
 Let us focus on the left boundary point $l$ (the same considerations apply to the right boundary $r$). Our assumption that~$X$ is conservative implies that $l$ can not be an exit-not-entrance boundary for~$X$, thus it can not be an entrance-not-exit boundary for $\widetilde X$. From the boundary classification presented on page \pageref{def:R_Q}, it follows that
\begin{itemize}\itemsep=0pt
\item[(i)] if $l$ is a non-singular boundary for $X$, then it is also a non-singular boundary for $\widetilde X$, and in this case we have a reflecting (killing) boundary condition for $X$ (respectively, $\widetilde X$);
\item[(ii)] if $l$ is a natural boundary for $X$, then it stays a natural boundary for $\widetilde X$;
\item[(iii)] if $l$ is an entrance-not-exit, it becomes an exit-not-entrance boundary for $\widetilde X$ (consistent with our definition of $\widetilde X$ as killed on the first exit from $(l,r)$).
\end{itemize}

We first establish this result in the special case $b(x) \equiv 0$. In this case, the process $X$ is in natural scale, i.e., $s_X(x) \equiv x$. Let $\lambda>0$ and let $\psi_{\lambda}$ and $\varphi_{\lambda}$ be the fundamental solutions to~${\L_X f=\lambda f}$. We claim that the fundamental solutions to $\L_{\widetilde X} f= \lambda f$ are given by
\begin{equation}\label{eqn:tilde_psi_phi}
\tilde \psi_{\lambda}(x)=\psi'_{\lambda}(x), \qquad
\tilde \varphi_{\lambda}(x)=-\varphi'_{\lambda}(x).
\end{equation}
To prove this, we note first that if $f$ is solves $\L_X f=\lambda f$, then $\tilde f=f'$ solves $\L_{\widetilde X} \tilde f=\lambda \tilde f$, due to the intertwining relation
\[
\frac{\d}{\d s(x)} \L_X = \L_{\widetilde X}  \frac{\d}{\d s(x)},
\]
which follows immediately from \eqref{L_tilde_L_speed_scale}.
 It is clear that $\tilde \psi_{\lambda}$ and $\tilde \varphi_{\lambda}$ are positive on $(l,r)$ (since~$\psi_{\lambda}$ is increasing
and $\varphi_{\lambda}$ is decreasing). Next, since $\lambda>0$ and $\psi_{\lambda}$ is a positive solution to $\tfrac{1}{2} \sigma^2(x) f''(x)=\lambda f(x)$, we have $\psi''_{\lambda}(x)>0$, so \smash{$\tilde \psi_{\lambda}=\psi'_{\lambda}$} is increasing. The same argument shows that $\tilde \varphi_{\lambda}$ is decreasing. According to
\cite[Section II.10]{Borodin2002}, in each case (natural, entrance-not-exit, non-singular reflecting boundary) we have
$\psi'(l+)=0$, which implies the correct killing boundary condition~${\tilde \psi_{\lambda}(l+)=0}$. The same considerations apply to the boundary condition $\tilde \varphi_{\lambda}$ at the right boundary $r$. This ends the proof of \eqref{eqn:tilde_psi_phi}.

Next, we check that if $\omega_{\lambda}=\wr[\varphi_{\lambda}; \psi_{\lambda}]$, then
\begin{equation}\label{new_Wronskian}
\wr\big[\tilde \varphi_{\lambda}; \tilde \psi_{\lambda}\big]=
-\varphi'_{\lambda}(x) \psi''_{\lambda}(x)+\varphi''_{\lambda}(x)\psi'_{\lambda}(x)
=\frac{2\lambda}{\sigma^2(x)} \wr[\psi_{\lambda}; \varphi_{\lambda}]=\lambda \omega_{\lambda} s_{\widetilde X}'(x),
\end{equation}
where in the second step we used the fact that $\psi_{\lambda}$ and $\varphi_{\lambda}$ are solutions to $\tfrac{1}{2} \sigma^2(x) f''(x)=\lambda f(x)$.
Using \eqref{Green_X} and \eqref{new_Wronskian} we can write the Green's function of the process \smash{$\widetilde X$} as
\begin{equation}\label{Green_tilde_X}
 G^{\widetilde X}_{\lambda}(x,y)=
 \begin{cases}
 -(\lambda \omega_{\lambda})^{-1} \psi'_{\lambda}(x) \varphi'_{\lambda}(y)  & {\text{if }} x\le y, \\
 -(\lambda \omega_{\lambda})^{-1} \psi'_{\lambda}(y) \varphi'_{\lambda}(x)  & {\text{if }} y\le x.
 \end{cases}
\end{equation}

Now we have all the ingredients to complete the proof of \eqref{eqn:main}.
We rewrite
\eqref{eqn:main} in the form
\begin{gather}\label{eqn:main3}
\int_l^y p^X_t(x_1,u) \d u=\int_l^y p^X_t(x_2,u) \d u+\int_{x_1}^{x_2} p^{\widetilde X}_t(y,v) \d v.
\end{gather}
We need to prove that \eqref{eqn:main3} holds for all $t>0$ and $x_1$, $x_2$, $y$ in $(l,r)$ with $x_1<x_2$. We take~${\lambda>0}$, multiply both sides of
\eqref{eqn:main3} by $\exp(-\lambda t)$ and integrate in $t$ over $(0,\infty)$. After applying Fubini's theorem, we obtain
\begin{equation}\label{eqn:proof_main2_1}
 2 \int_l^y G^X_{\lambda}(x_1,u)
 \sigma^{-2}(u)\d u
 = 2 \int_l^y G^X_{\lambda}(x_2,u)
 \sigma^{-2}(u)\d u +\int_{x_1}^{x_2} G^{\widetilde X}_{\lambda}(y,v) \d v.
\end{equation}
By the uniqueness of the Laplace transform, to prove \eqref{eqn:main3}, it is enough to show that \eqref{eqn:proof_main2_1} holds for all $\lambda>0$ and $x_1,x_2,y \in (l,r)$ such that $x_1<x_2$. Identity \eqref{eqn:proof_main2_1} is true if
\begin{equation}\label{eqn:proof_main2_2}
 2 \partial_x \int_l^y G^X_{\lambda}(x,u) \sigma^{-2}(u)\d u + G^{\widetilde X}_{\lambda}(x,y)=0,
\end{equation}
for $x,y \in (l,r)$. We first prove that \eqref{eqn:proof_main2_2} holds when $l<x<y$. Using \eqref{Green_X} and \eqref{Green_tilde_X},
we rewrite \eqref{eqn:proof_main2_2} in the form
\begin{gather}
\frac{2}{\omega_{\lambda}}
\partial_x \left[ \varphi_{\lambda}(x) \int_l^x \psi_{\lambda}(u) \sigma^{-2}(u) \d u+ \psi_{\lambda}(x) \int_x^y \varphi_{\lambda}(u) \sigma^{-2}(u) \d u \right]
\nonumber\\
\qquad- \frac{1}{\lambda \omega_{\lambda}} \psi'_{\lambda}(x) \varphi'_{\lambda}(y)=0.\label{eqn:proof_main2_3}
\end{gather}
Since $\psi_{\lambda}$ solves $\tfrac{1}{2} \sigma^2(x) f''(x)=\lambda f(x)$, we have
\[
2\int_l^x \psi_{\lambda}(u) \sigma^{-2}(u) \d u=\frac{1}{\lambda}
\int_l^x \psi''_{\lambda}(u) \d u=\frac{1}{\lambda} \psi'_{\lambda}(x),
\]
where in the last step we used $\psi'_{\lambda}(l+)=0$ established above. Similarly,
\[
2 \int_x^y \varphi_{\lambda}(u) \sigma^{-2}(u) \d u=\frac{1}{\lambda}
\bigl( \varphi'_{\lambda}(y)-\varphi'_{\lambda}(x)\bigr).
\]
With the help of the above two equations, \eqref{eqn:proof_main2_3} becomes
\[
\frac{1}{\lambda \omega_{\lambda}}
\partial_x \bigl( \wr[\varphi_{\lambda}; \psi_{\lambda}] +
\psi_{\lambda}(x) \varphi'_{\lambda}(y) \bigr) - \frac{1}{\lambda \omega_{\lambda}}
\psi'_{\lambda}(x) \varphi'_{\lambda}(y)=0,
\]
and this identity is true since $\wr[\varphi_{\lambda}; \psi_{\lambda}]=\omega_{\lambda}$ does not depend on $x$. Thus, we have proved~\eqref{eqn:proof_main2_2} when $l<x<y$. The proof that \eqref{eqn:proof_main2_2} holds when $l<y< x$ is done in the same way, and the details are omitted.

This ends the proof of \eqref{eqn:main} in the driftless case $b(x)\equiv 0$. The general case follows by considering the driftless process $Y=\{s(X_t)\}_{t\ge 0}$, constructing the transformed process $\widetilde Y$, and defining \smash{$\widetilde X_t=s^{-1}\bigl(\widetilde Y_t\bigr)$}. The fact that
identity \eqref{eqn:main} holds for processes $Y$ and $\widetilde Y$ implies that it also holds for $X$ and $\widetilde X$. We leave it as an exercise to check that the process $\widetilde X$ thus constructed has the infinitesimal generator as in
\eqref{def_tilde_L2}.
\end{proof}

Identity \eqref{eqn:main} is equivalent to saying that for any fixed $t>0$ and $y \in (l,r)$ the function
\begin{equation}\label{const_function}
x\mapsto \p_x (X_t\le y) - \p_y\bigl(\widetilde X_t>x\bigr)
\end{equation}
is constant (i.e., does not depend on $x$). Note that this is very similar to classical Siegmund duality \eqref{Siegmund1}, where this constant is equal to zero. Taking derivatives in $x$ of \eqref{const_function} and relabelling variables $x\leftrightarrow y$, we obtain the following result.
\begin{Corollary}\label{corollary_main}
The transition probability densities of $X$ and $\widetilde X$ satisfy
\begin{equation}\label{eqn:main2}
p^{\widetilde X}_t(x,y)=\partial_y \p_y(X_t>x)=\partial_y \int_x^r p^X_t(y,u) \d u,
\end{equation}
for $t>0$ and $x,y\in (l,r)$.
\end{Corollary}

\section{Darboux transform of killed Brownian motion}\label{section_Darboux_diffusion}

We start with a diffusion process $Y$ on the interval $(l,r) \subseteq \r$, which is a Brownian motion killed at rate $c(y)$, where $c(y)$ is a nonnegative continuous function for $y\in (l,r)$. The infinitesimal generator is $\L_Y = \frac{1}{2} \partial_y^2 -c(y)$. We take the speed measure to be the Lebesgue measure~${\d m(x)=\d x}$, in this case the scale function becomes $s(x)=2x$ (see \eqref{eqn:s_m_k_prime}). If $l=-\infty$, then $l$ is a~natural boundary. If $l$ is finite, then the boundary classification conditions $R \in L_1((l,z))$ and~${Q \in L_1((l,z))}$ from page \pageref{R_Q_L1} are equivalent to
\[
R \in L_1((l,z)) \Leftrightarrow \int_l^z c(y) (y-l) \d y < \infty\qquad  \text{and}\qquad
Q \in L_1((l,z)) \Leftrightarrow \int_l^z c(y) \d y < \infty.
\]
Thus, if a boundary is an entrance, it must also be an exit, and we have the following three possibilities for boundary behavior of the process $Y$: natural, exit-not-entrance and non-singular (entrance-not-exit boundary is not possible for a killed Brownian motion). We will only consider killed Brownian motion processes that have one of elastic/reflecting/killing boundary conditions at each non-singular boundary point (see the discussion on page \pageref{elastic_boundary_condition}).

Next, we assume that for some $\lambda \in \r$ we have found a function $h \colon (l,r) \mapsto (0,\infty)$ that is $\lambda$-invariant for the process $Y$, that is
\begin{equation}\label{h_lambda_invariant}
\e_x[h(Y_t)]={\rm e}^{\lambda t} h(x), \qquad x\in (l,r), \quad t>0.
\end{equation}
This $\lambda$-invariant function $h$ must satisfy the following properties:
\begin{itemize}\itemsep=0pt
\item[(i)] $h$ is a solution to $\L_Y h = \lambda h$ on $(l,r)$;
\item[(ii)]
$h$ satisfies appropriate boundary conditions at each non-singular boundary point.
\end{itemize}
The first property follows from the fact that the process $\exp(-\lambda t) h(Y_t)$ is a martingale. To check the second property, we take $\mu> \max(0,\lambda)$, multiply both sides of~\eqref{h_lambda_invariant} by ${\rm e}^{-\mu t}$, integrate in~${t \in (0,\infty)}$ and obtain
\begin{equation}\label{h_lambda_invariant2}
(\mu - \lambda) \int_l^r G^Y_{\mu}(x,y) h(y) \d y= h(x),
\end{equation}
where $G^Y_{\mu}(x,y)$ is the Green's function of the process $Y$. Thus, if $h$ is bounded on $(l,r)$ (which happens, for example, if both boundaries are non-singular), it must be in the domain of the infinitesimal generator $\L_Y$ (because it lies in the image of the resolvent operator). Therefore, $h$ must satisfy the appropriate boundary conditions at each non-singular boundary (see \cite[Section~II.7]{Borodin2002}). The proof of~(ii) in the general case, when $h$ can be unbounded, is given in Appendix~\ref{AppendixB}.

In what follows, we will require the following.

\begin{Assumption}\label{Assumption_1}
The constant
\begin{equation}\label{def_m_h}
{\mathfrak m}_h:= \sup\limits_{y\in (l,r)} \left[ c(y)- \left(\frac{h'(y)}{h(y)}\right)^2 \right]
\end{equation}
is finite.
\end{Assumption}

\begin{Remark}
 Note that ${\mathfrak m}_h$ is finite when $c(y)$ is bounded on $(l,r)$. In all our examples in Section~\ref{section_Examples}, we actually have ${\mathfrak m}_h=0$, though it is easy to construct examples with ${\mathfrak m}_h$ equal to any positive number (for instance, one could take Brownian motion on $\r$ killed at a positive constant rate and a
 $\lambda$-invariant function $h(y)=\cosh(y)$, see Section~\ref{example1}). It may be true that~${\mathfrak m}_h$ is always finite, but we could not prove this or find a counterexample.
\end{Remark}

As the first step of constructing the Darboux transform of $Y$, we introduce a new Markov process $X$ via Doob's $h$-transform
\begin{equation}\label{X_as_Doobs_transform}
\p_x (X_t \in A)=\frac{{\rm e}^{-\lambda t}}{h(x)}
\e_x[ {\mathbf 1}_{\{Y_t \in A\}} h(Y_t)].
\end{equation}
The process $X$ is a conservative Markov process, since $h$ is $\lambda$-invariant for $Y$. Moreover, it is a~diffusion process on $(l,r)$ with infinitesimal generator
\begin{equation}\label{eqn:L_X}
\L_X=\frac{1}{h} \L_Y h-\lambda=\frac{1}{2} \partial_y^2 + \frac{h'(y)}{h(y)} \partial_y.
\end{equation}
The derivatives of the scale function and the speed measure of the process $X$ are $s'(x)=2h^{-2}(x)$ and $m'(x)=h^2(x)$ (this follows from \eqref{eqn:s_m_k_prime}), thus the functions $R$ and $Q$ in \eqref{def:R_Q} are given by
\begin{equation}\label{R_X_Q_X}
R_X(x)=2 h^{-2}(x) \int_x^z h^2(y) \d y, \qquad
Q_X(x)=2 h^{2}(x) \int_x^z h^{-2}(y) \d y.
\end{equation}
The functions $R_X$ and $Q_X$ can be used to find the type of the boundary points $l$ and $r$ for the process $X$ (see Section~\ref{section_preliminaries}).

To proceed, we will require the following.
\begin{Assumption}\label{Assumption_2}
The process $X$ has reflecting boundary condition at each non-singular boundary point.
\end{Assumption}

The next proposition shows that the above assumption is satisfied whenever $Y$ has a non-singular boundary with a reflecting or elastic boundary condition. This result will be useful for our examples in Section~\ref{section_Examples}.

\begin{Proposition}\label{prop_reflecting}
If $l$ is a non-singular boundary point of $Y$ with a reflecting or elastic boundary condition, then $h(l+)>0$ and $l$ is also
non-singular for $X$ with a reflecting boundary condition. The same statement applies to the right boundary point $r$.
\end{Proposition}
\begin{proof}
Assume that $l$ is a non-singular boundary point for $Y$ with reflecting or elastic boundary condition, which can be written as $g'(l+)=\gamma g(l+)$ for some $\gamma \ge 0$ (here we used the fact that~$Y$ is in the natural scale, so that $g^+(y)=g'(y)$). The function $h$, being $\lambda$-invariant for $Y$, must satisfy the equation $\L_Y h = \lambda h$ and the boundary condition $h'(l+)=\gamma h(l+)$. Since~$l$ is a non-singular boundary point with reflecting or elastic boundary condition, it is possible to start the process $X$ at $l$ (see \cite[Section II.6]{Borodin2002}). Then condition \eqref{h_lambda_invariant} also holds in the limit as $x \to l+$, so that $h(l+)={\rm e}^{-\lambda t} \e_l[h(X_t)]>0$ (since $h$ is positive on $(l,r)$). Next, we use the facts that $h$ is continuous and positive on $(l,r)$ and satisfies $h(l+)>0$ and check that both functions $R_X$ and~$Q_X$ defined in \eqref{R_X_Q_X} lie in $L_1((l,z))$. This implies that $l$ is a non-singular boundary for $X$.

Denote $\eta=\max(0,-\lambda)$ and define the process $Z$ as equal to $X$ killed a constant rate $\eta$ ($Z$ and $X$ are identical when $\eta=0$). From \eqref{X_as_Doobs_transform}, we see that for all $t>0$ and $x,y \in (l,r)$
\begin{equation}\label{eqn_p_t_Z_p_t_X}
p_t^Z(x,y)={\rm e}^{-\eta t} p_t^X(x,y)=\frac{{\rm e}^{-(\eta+\lambda) t}}{h(x)} p_t^Y (x,y) h(y).
\end{equation}
The left boundary $l$ is clearly non-singular for $Z$ and the boundary behaviour of $Z$ and $X$ is identical (a function $f\in C^2((l,r))$ is in the domain of the infinitesimal generator ${\mathcal L}_Z$ if and only if it is in the domain of the infinitesimal generator ${\mathcal L}_X={\mathcal L}_Z+\eta$).
The Green's function of $Z$ (with respect to the speed measure $\d m(x)=h^2(x) \d x$) is given by
\begin{equation}\label{Green_function_X}
G^Z_{\mu}(x,y)=\frac{G^Y_{\mu+\kappa}(x,y)}{h(x) h(y)},
\end{equation}
where $\mu>0$, $x,y \in (l,r)$ and we denoted $\kappa:=\eta+\lambda=\max(0,\lambda)$. This follows from \eqref{eqn_p_t_Z_p_t_X}, by multiplying both sides by $\exp(-\mu t)$ and integrating in $t \in (0,\infty)$. Comparing \eqref{Green_X} and \eqref{Green_function_X}, we see that the fundamental increasing and decreasing solutions $\psi^Z_{\mu}$ and $\varphi^Z_{\mu}$ for the operator $\L_Z$ must be given by
\[
\psi^Z_{\mu}(x)=\frac{\psi^Y_{\mu+\kappa}(x)}{h(x)}, \qquad
\varphi^Z_{\mu}(x)=\frac{\varphi^Y_{\mu+\kappa}(x)}{h(x)}, \qquad \mu>0,\quad  x\in (l,r).
\]

As we discussed in Section~\ref{section_preliminaries}, the boundary condition of $Z$ at $l$ can be deduced from the boundary behaviour of the increasing fundamental solution $\psi^Z_{\mu}(x)$ at $x=l+$. We know from \cite[Section II.10]{Borodin2002} that for all $\mu>0$ the function \smash{$g(y):=\psi^Y_{\mu+\kappa}(x)$} must satisfy the elastic or reflecting boundary condition $g'(l+)=\gamma g(l+)$ (the same one as satisfied by the $\lambda$-invariant function $h$). Denote $f(x):=\psi^Z_{\mu}(x)=g(x)/h(x)$. Then
\begin{align*}
f^+(l+)=\lim\limits_{x\to l+} \frac{f'(x)}{s'(x)}=
\frac{1}{2} \lim\limits_{x\to l+} h^2(x) f'(x)&=
\frac{1}{2} \lim\limits_{x\to l+} \big[
h(x) g'(x)-
g(x) h'(x)
\big]=0,
\end{align*}
since $h'(l+)=\gamma h(l+)$ and $g'(l+)=\gamma g(l+)$.
Thus the increasing fundamental solution $f(x)=\psi^Z_{\mu}(x)$ satisfies $f^+(l+)=0$ (for all $\mu>0$), which is a reflecting boundary condition. Therefore, the process $Z$ (and hence the process $X$) has a reflecting boundary at $l$.

When $r$ is a non-singular boundary point for $Y$, the proof proceeds along the same lines, except that now we focus on the boundary behavior at $r$ of the decreasing fundamental solution~$\varphi^X_{\mu}(x)$.
\end{proof}

\begin{Remark}
It is likely that Assumption \ref{Assumption_2} is always true. The intuitive argument for the validity of Assumption \ref{Assumption_2} is the following: the process $X$ is conservative by construction, thus it cannot have elastic or killing boundary conditions at non-singular boundary (as this would imply~${\p_x(X_t \in (l,r))<1}$). It can not have sticky boundaries or traps, as then the distribution of~$X_t$ would have an atom at that boundary (which is impossible since the law of $Y_t$ has no atoms). Thus, the only possible boundary conditions are reflecting or non-local boundary conditions (see~\cite[Theorem 2, p.\ 39]{Mandl1968}). The latter can be ruled out if we could show that $X$ has continuous paths, but we did not pursue this further, as Proposition \ref{prop_reflecting} was sufficient for our examples in Section~\ref{section_Examples}.
\end{Remark}

We proceed with the next step in our construction. The process $X$, constructed via Doob's $h$-transform \eqref{X_as_Doobs_transform}, is a conservative diffusion process, and we assumed that it has reflecting boundary condition at every non-singular boundary point. We construct its Siegmund transform~$\widetilde X$ via Theorem~\ref{thm:main}. According to \eqref{def_tilde_L2} and \eqref{eqn:L_X}, the infinitesimal generator of $\widetilde X$ is
\[
\L_{\widetilde X}=\frac{1}{2} \partial_y^2 - \frac{h'(y)}{h(y)} \partial_y.
\]
Now we want to turn it into a killed Brownian motion by ``removing the drift term". We achieve this through another Doob's $h$-transform. First we establish that the function $h$ is $\nu$-excessive for the process $\widetilde X$ with $\nu={\mathfrak m}_h+\lambda$.

\begin{Proposition}\label{prop_excessive}
For $x\in (l,r)$ and $t>0$ we have \smash{$\e_x\big[ h\bigl(\widetilde X_t\bigr)\big] \le {\rm e}^{ ({\mathfrak m}_h + \lambda) t} h(x) $}.
\end{Proposition}
\begin{proof}
The process $\widetilde X$ is a diffusion process solving the SDE
\[
\d \widetilde X_t=-\frac{h'\bigl(\widetilde X_t\bigr)}{h\bigl(\widetilde X_t\bigr)} \d t + \d W_t, \qquad t<\zeta,
\]
where $\zeta$ is the first time the process reaches the boundary of $(l,r)$. When both boundaries are natural, we have $\zeta=+\infty$, and if one of the boundaries is exit-not-entrance or non-singular killing boundary, then the process is killed when it hits that boundary \big(in other words, the process at time $\zeta$ is sent to a cemetery state $\widetilde X_{\zeta}=\Delta$\big).

Applying Ito's formula to $\ln h\bigl(\widetilde X_t\bigr)$ gives
\begin{equation}\label{eqn:Ito}
\ln h\bigl(\widetilde X_t\bigr)=\ln h (x) +
\int_0^t \frac{h'\bigl(\widetilde X_s\bigr)}{h\bigl(\widetilde X_s\bigr)} \d \widetilde X_t
+\frac{1}{2} \int_0^t v\bigl(\widetilde X_s\bigr) \d s, \qquad t<\zeta,
\end{equation}
where
\[
v(x):=\frac{\d^2}{\d x^2} \ln h(x)=\frac{h''(x)}{h(x)}-\left(\frac{h'(x)}{h(x)}\right)^2=2(c(x)+\lambda)-\left(\frac{h'(x)}{h(x)}\right)^2.
\]
In the last step, we used the fact that $h$ is a solution to $\L_Y h=\lambda h$.

Let us denote
\[
U_t:=\frac{h'\bigl(\widetilde X_t\bigr)}{h\bigl(\widetilde X_t\bigr)}, \qquad Z_t:=\int_0^{t} U_s  \d W_s, \qquad t<\zeta.
\]
After rearranging the terms in \eqref{eqn:Ito} and using the fact that $h\bigl(\widetilde X_{t}\bigr)=h(\Delta)=0$ for $t\ge \zeta$, we obtain
\begin{equation}\label{eqn:Ito2}
h\bigl(\widetilde X_t\bigr) \le h(x) \times {\rm e}^{ Z_{t\wedge \zeta} - \frac{1}{2} \langle Z \rangle_{t\wedge \zeta} + \frac{1}{2} \int_0^{t\wedge \zeta} (v(\widetilde X_s) - U_s^2)\d s},
\end{equation}
which holds for all $t\ge 0$.

Next, we observe the following two facts. First, the process $\exp( Z_{t\wedge \zeta} - (1/2)\langle Z \rangle_{t\wedge \zeta})$ is a~positive local martingale; thus, it is a supermartingale, and its expected value is at most one. Second, we have the bound
\[
\frac{1}{2} \int_0^{t\wedge \zeta} \bigl(v\bigl(\widetilde X_s\bigr) - U_s^2\bigr) \d s =
\int_0^{t\wedge \zeta} \left[ c\bigl(\widetilde X_s\bigr) + \lambda - \left( \frac{h'\bigl(\widetilde X_s\bigr)}{h\bigl(\widetilde X_s\bigr)}\right)^2 \right] \d s \le ({\mathfrak m}_h+\lambda) t,
\]
which follows from \eqref{def_m_h}. Using these two facts and taking expectations in \eqref{eqn:Ito2} gives the desired result.
\end{proof}

Now that we established the fact that $h$ is $\nu$-excessive for the process $\widetilde X_t$ (with $\nu={\mathfrak m}_h + \lambda$), we can define a new process $\widetilde Y$ via Doob's $h$-transform
\[
\p_y (\widetilde Y_t \in A)=\frac{{\rm e}^{-({\mathfrak m}_h+\lambda) t}}{h(y)}
\e_y\big[ {\mathbf 1}_{\{\widetilde X_t \in A\}} h\bigl(\widetilde X_t\bigr)\big].
\]
The infinitesimal generator of $\widetilde Y$ is
\[
\L_{\widetilde Y}=\frac{1}{h} \L_{\widetilde X} h - ({\mathfrak m}_h+\lambda)
=\frac{1}{2} \partial_y^2 - \tilde c(y),
\]
where
\begin{equation}\label{new_tilde_c}
\tilde c(y)={\mathfrak m}_h+\left(\frac{h'(y)}{h(y)}\right)^2 - c(y).
\end{equation}
Note that $\tilde c(y)$ is a nonnegative continuous function on $(l,r)$. This is due to the way we defined~${\mathfrak m}_h$ in \eqref{def_m_h} and due to Assumption \ref{Assumption_1}. Thus we can identify $\widetilde Y$ as a Brownian motion killed at rate $\tilde c(y)$.

\begin{Definition}\label{def_Darboux_transform}
{\textnormal{
We call the process $\widetilde Y$ constructed above {\it the Darboux transform} of the killed Brownian motion process $Y$. The positive $\lambda$-invariant function $h$ used in this construction is called {\it the seed function}. }}
\end{Definition}

\begin{table}[t!]\centering
\begin{tabular}{ c c c }
 process $Y$ & $\qquad \qquad \qquad \qquad $ & process $\widetilde Y$ \\
 $\L_Y=\frac{1}{2} \partial_y^2 - c(y)$ & & $\L_{\widetilde Y}=\frac{1}{2} \partial_y^2 - \tilde c(y)$ \\
 $\lambda$-invariant function $h(y)$ & & $\tilde c(y)={\mathfrak m}_h+\bigl(\frac{h'(y)}{h(y)} \bigr)^2 - c(y)$\\
 \rule{0pt}{15pt}
 $\downarrow$ & & $\uparrow$ \\
 \rule{0pt}{20pt}
 Doob's $h$-transform of $Y$ & & Doob's $h$-transform of $\widetilde X$ \\
 \rule{0pt}{15pt}
 $\downarrow$ & & $\uparrow$ \\
 \rule{0pt}{20pt}
 process $X$ & $\rightarrow \quad $ Siegmund dual of $X$
 $\quad \rightarrow $ & process $\widetilde X$ \\
 $\L_X=\frac{1}{2} \partial_y^2 + \frac{h'(y)}{h(y)} \partial_y$ & (see Theorem~\ref{thm:main}) & $\L_{\widetilde X}=\frac{1}{2} \partial_y^2 - \frac{h'(y)}{h(y)} \partial_y$ \\
\end{tabular}

\caption{The three steps in constructing Darboux transform of killed Brownian motion process $Y$.}\label{table:1}
\end{table}

We summarize the steps in our construction of the Darboux transformed process $\widetilde Y$ in Table~\ref{table:1}. We want to emphasize that this construction depends on the positive $\lambda$-invariant function $h$ and requires Assumptions
\ref{Assumption_1} and \ref{Assumption_2}.
In practical applications, we find a $\lambda$-invariant function $h$ by solving the equation $\L_Y h =\lambda h$ with appropriate boundary conditions at each non-singular boundary. These two conditions alone do not guarantee that $h$ is $\lambda$-invariant (see \cite{Rogers_2021, Mijatovic_2012}); thus, to verify that this candidate function $h$ is indeed $\lambda$-invariant, we check that \eqref{h_lambda_invariant} holds by actually computing the left-hand side (an alternative way is to apply \cite[Theorem 2.7]{Rogers_2021} or~\cite[Corollary 2.2]{Mijatovic_2012}). The condition ${\mathfrak m}_h< \infty$ of Assumption \ref{Assumption_1} is easy to verify directly (once we have expressions for $c(x)$ and $h(x)$). Assumption \ref{Assumption_2} is covered by Proposition \ref{prop_reflecting} in most cases of interest.

Next, we present our main result in this section, which connects transition probability densities of the process $Y$ and its Darboux transform $\widetilde Y$.

\begin{Theorem}\label{thm:main2} Let $Y$ be a killed Brownian motion on $(l,r)$ that has one of elastic/reflecting/kil\-l\-ing boundary conditions at each non-singular boundary point. Assume that $h$ is a positive $\lambda$-invariant function for $Y$ and that both Assumptions {\rm\ref{Assumption_1}} and
{\rm\ref{Assumption_2}} are satisfied. Let $\widetilde Y$ be the Darboux transform of $Y$, constructed with the seed function $h$. Then for $t>0$, $x,y\in (l,r)$ we have
\begin{equation}\label{equation_tilde_p}
 p^{\widetilde Y}_t(x,y)={\rm e}^{-({\mathfrak m}_h + 2\lambda) t}\frac{h(x)}{h(y)} \partial_{x}
\left[ \frac{1}{h(x)} \int_{y}^r p^{Y}_t(x,u) h(u) \d u \right].
\end{equation}
If $Y$ has a non-singular boundary point with a reflecting or elastic boundary condition, then this point is also non-singular for $\widetilde Y$ with a killing boundary condition.
\end{Theorem}

\begin{proof}
Since $X$ \big(respectively, $\widetilde Y$\big) is the Doob's $h$-transform of $Y$ \big(respectively, $\widetilde X$\big), their transition probability densities satisfy
\begin{align*}
 p^X_t(x,y)=\frac{{\rm e}^{-\lambda t}}{h(x)}p_t^Y(x,y) h(y), \qquad
 p^{\widetilde Y}_t(x,y)=\frac{{\rm e}^{-({\mathfrak m}_h + \lambda) t}}{h(x)}p_t^{\widetilde X}(x,y) h(y).
\end{align*}
The transition probability densities of $\widetilde X$ and $X$ are related via \eqref{eqn:main2}. Combining these identities and using the symmetry of \smash{$p_t^{\widetilde Y}(x,y)$} with respect to $x$ and $y$, we obtain \eqref{equation_tilde_p}.

Assume that $l$ is a non-singular boundary points for $Y$ with a reflecting or elastic boundary condition. According to Proposition
\ref{prop_reflecting}, we have $h(l+)>0$ and this boundary point is also non-singular for $X$ with a reflecting boundary condition. Then its Siegmund dual $\widetilde X$ has a killing boundary condition (see Theorem~\ref{thm:main}), which is preserved when we construct the process $\widetilde Y$ via Doob's $h$-transform of $\widetilde X$. The last statement uses the fact that $h(l+)>0$ and can be proved by exactly the same argument as was used in the proof of Proposition \ref{prop_reflecting}.
\end{proof}

Theorem~\ref{thm:main2} states exact conditions under which our informal construction presented in Section~\ref{section_intro} yields a correct (up to an exponential factor) transition probability density of a diffusion process. That informal construction relied on the intertwining relation between the generators~\eqref{L_tilde_L_intertwining}, and we hypothesised that it can be extended to an intertwining relation on the corresponding Markov semigroups \eqref{def_tildeP}. We would like to point out the references
\cite{Bonnefont_2014,Fill_2016,Patie_2019}, where intertwining was applied to semigroups of diffusion processes. It would be interesting to see if a more direct proof of Theorem~\ref{thm:main2} could be given, where the intertwining relation between semigroups would play a prominent role. We leave this question to future work.

{\bf Importance of $\boldsymbol{\lambda}$-invariance of $\boldsymbol{h}$.}
 As we discussed in the introduction, to define the Darboux transform of operator $\L$ we need a positive function $h$ that satisfies $\L h=\lambda h$. If $h$ is $\lambda$-invariant for the process $Y$, then $h$ necessarily satisfies $\L_Y h=\lambda h$. However, the converse is not true: a function satisfying $\L_Y h=\lambda h$ may fail to satisfy the boundary conditions and thus would not be a $\lambda$-invariant function. The $\lambda$-invariance condition on $h$ is very important for our construction of Darboux transformed process and in general it can not be relaxed.

 As an example, consider a process $Y$, which is a Brownian motion on $(0,\infty)$ reflected at zero. The transition probability of $Y$ (with respect to Lebesgue measure) is
\[
p_t^Y(x,y)=\frac{1}{\sqrt{2\pi t}} \bigl({\rm e}^{-\frac{1}{2t}(y-x)^2}+{\rm e}^{-\frac{1}{2t}(y+x)^2}\bigr), \qquad t,x,y>0.
\]
The infinitesimal generator is $\L_Y=\frac{1}{2} \partial_y^2$. The point $l=0$ is a non-singular boundary and we have a boundary condition $f'(0+)=0$. Consider a function $h(y)=y$. This function is a solution to $\L_Y h = 0$, but it is not an invariant function for $Y$, since $h$ does not satisfy the reflecting boundary condition at zero. If we take the above expression for $p^Y_t(x,y)$ and compute the expression in the right-hand side of \eqref{equation_tilde_p} (with $h(y)=y$), we would obtain
\begin{gather*}
\frac{h(x)}{h(y)} \partial_{x}
\left[ \frac{1}{h(x)} \int_{y}^r p^{Y}_t(x,u) h(u) \d u \right]\\
\qquad= \frac{1}{\sqrt{2\pi t}}
\left[{\rm e}^{-\frac{1}{2t}(y-x)^2}\left(1-\frac{t}{xy}\right)-{\rm e}^{-\frac{1}{2t}(y+x)^2}\left(1+\frac{t}{xy}\right)\right].
\end{gather*}
The expression in the right-hand side is negative for small values of $y$, thus it can not be a~transition probability density. This example confirms that $\lambda$-invariance of $h$ is a condition that can not be relaxed.

{\bf A connection with Krein dual strings.}
The second step in our construction of the Darboux transformed process (see Table \ref{table:1}) is closely related to the concept of Krein dual strings, see \cite{Comtet_2011,Kotani_1982}. To demonstrate this connection, we introduce two increasing functions $u(x)$ and $v(x)$ via
$u'(x)=h^2(x)$ and $v'(x)=2 h^{-2}(x)$
and define diffusion processes $V_t=v(X_t)$ and $U_t=u\bigl(\widetilde X_t\bigr)$. These are processes in natural scale (the functions $v$ and $u$ are scale functions for $X$ and \smash{$\widetilde X$}), and their infinitesimal generators can be written as follows:
\[
\L_U=\frac{\d}{\d M(u)} \frac{\d}{\d u}, \qquad \L_V=\frac{\d}{\d m(v)} \frac{\d}{\d v},
\]
where $M$ and $m$ are speed measures of $U$ and $V$ (see \eqref{eqn:s_m_k_prime})
\[
\frac{\d m}{\d v}=\frac{1}{2} h^{4}(x(v)), \qquad
\frac{\d M}{\d u}=2 h^{-4}(x(u)).
\]
Here $x(u)$ and $x(v)$ are the inverse functions of $u(x)$ and $v(x)$. One can check that
$
\frac{\d M}{\d m} \frac{\d m}{\d v}=1$,
which means that $m(v)$ is the inverse function of $M(u)$. This shows that $m$ is the Krein dual string of $M$, in the terminology of \cite{Comtet_2011,Kotani_1982}.

\section{Examples}\label{section_Examples}

In this section, we present five examples of Darboux transformed processes. In the first four examples, we take the process $Y$ to be the Brownian motion on an interval $(l,r)$ with various boundary conditions, and our last example is related to the Ornstein--Uhlenbeck process. For all killed Brownian motion processes in this section, we take the speed measure $\d m(x)=\d x$, so that the transition probability density and the Green's function for each process are given with respect to the Lebesgue measure.

Our goal in this section is to compute the transition probability density of the transformed process $\widetilde Y$ and (in most cases) to provide its spectral expansion. Computing \smash{$p_t^{\widetilde Y}(x,y)$} via \eqref{equation_tilde_p}
is very time-consuming if done by hand. Instead, we obtained all expressions for \smash{$p_t^{\widetilde Y}(x,y)$} in this section using symbolic computations, and then we verified numerically that our formulas were correct. The {\sc Matlab} programs for verifying these formulas via symbolic computations can be found at \href{https://kuznetsovmath.ca/}{kuznetsovmath.ca}.

Some computations in this section will require the following simple result (which is probably well known, but we could not find it in the literature in this exact form).
\begin{Lemma}\label{lemma_Dh}
Consider a second-order linear differential operator $\L=\frac{1}{2} \partial_y^2-c(y)$, where $c$ is a~continuous function of $y\in (l,r)$. Assume that $h$, $f$ and $g$ are $C^2$ functions on $(l,r)$ such that
\begin{itemize}\itemsep=0pt
\item[$(i)$] $h$ is positive and satisfies $\L h=\lambda h$ for $y\in (l,r)$;
\item[$(ii)$] $g$ satisfies $\L g=\mu g$ for $y\in (l,r)$.
\end{itemize}
Denote $\tilde f={\mathcal D}_h f$ and $\tilde g=\mathcal D_h g$. Then
\begin{equation}\label{Prop1:eqn2}
\int \tilde f(y) \tilde g(y) \d y=f(y) \tilde g(y) +2(\lambda-\mu) \int f(y) g(y) \d y.
\end{equation}
If $f$ also satisfies $\L f=\mu f$ for $y\in (l,r)$, then
 \begin{equation}\label{Prop1:eqn3}
\wr\big[\tilde f, \tilde g\big]= 2 (\lambda - \mu) \wr[f,g].
\end{equation}
\end{Lemma}
The proof of \eqref{Prop1:eqn3} is obtained by calculating $\wr\big[\tilde f, \tilde g\big]$ and using the fact that $g''=2(c+\mu) g$ and $h''=2(c+\lambda) h$. Formula \eqref{Prop1:eqn2} is derived in a similar way, first by taking derivative in $y$ of both sides and then simplifying the result using the above identities for $g''$ and $h''$.

We also record here the following useful fact \big(which follows from the intertwining relation~\eqref{L_tilde_L_intertwining} and our construction of Darboux transformed process $\widetilde Y$\big)
\begin{equation}\label{g_f_intertwining}
  \L_Y f=(\mu+ {\mathfrak m}_h + 2\lambda) f
 \quad \Longrightarrow \quad \L_{\widetilde Y} ({\mathcal D}_h f)=\mu ({\mathcal D}_h f).
\end{equation}

\subsection[Brownian motion on R]{Brownian motion on $\boldsymbol{\r}$}\label{example1}

We take $Y$ to be a standard Brownian motion on $\r$. In this case, we have $l=-\infty$, $r=+\infty$, both boundaries are natural and the transition probability density of $Y$ is
\[
p^Y_t(x,y)=\frac{1}{\sqrt{2\pi t}} {\rm e}^{-\frac{1}{2t} (y-x)^2},\qquad t>0, \quad x,y \in \r.
\]
The infinitesimal generator is $\L_Y=\frac{1}{2} \partial_y^2$, thus the killing term $c(x)$ is zero.
We set $\lambda=1/2$ and take $h(x)=\cosh(x)$. Note that $h$ solves the equation $\L_Y h= \lambda h$. To check that $h$ is indeed $\lambda$-invariant for $Y$, we verify that
\begin{equation}\label{lambda_invariant_h}
\int^r_{l} p^Y_t(x,y) h(y) \d y={\rm e}^{\lambda t} h(x).
\end{equation}
The integral in the left-hand side can be easily computed explicitly. We also provide the {\sc Matlab} code for symbolic verification of this result. Another way to prove that $h$ is $\lambda$-invariant is via~\cite[Theorem 2.7]{Rogers_2021} or \cite[Corollary 2.2]{Mijatovic_2012}.

Next, we verify that Assumption \ref{Assumption_1} holds with ${\mathfrak m}_h=0$ and Assumption \ref{Assumption_2} is also satisfied since both boundaries are natural for the process $X$ with the generator given in \eqref{eqn:L_X}. Thus, we can construct the Darboux transformed process $\widetilde Y$ and identify it as a Brownian motion on~$\r$ killed at rate $\tilde c(y)=\tanh(y)^2$ (this latter expression follows from \eqref{new_tilde_c}).
The transition probability density of $\widetilde Y$ is given by
\begin{align}
p^{\widetilde Y}_t(x,y)={}&\frac{1}{\sqrt{2\pi t}} {\rm e}^{-t-\frac{1}{2t}(y-x)^2}\nonumber\\
&
+ \frac{{\rm e}^{-\frac{t}{2}}}{2\cosh(x) \cosh(y)} \left[
\Phi\left( \frac{y-x+t}{\sqrt{t}} \right) - \Phi\left( \frac{y-x-t}{\sqrt{t}} \right)\right],\label{p_tildeY1}
\end{align}
where $\Phi(x)$ is the CDF of standard normal distribution. Formula \eqref{p_tildeY1} was obtained from
\eqref{equation_tilde_p} by symbolic computation (and then verified numerically).

Now we turn our attention to the Green's function of $\widetilde Y$. The fundamental increasing/de\-creas\-ing solutions of $\L_Y f= \mu f$ (with $\mu>0$) are $\psi_{\mu}^Y(y)=\exp\bigl(\sqrt{2\mu} y\bigr)$ and
$\varphi_{\mu}^Y(y)=\exp\bigl(-\sqrt{2\mu} y\bigr)$. Let us denote
$
F(z,y)={\mathcal D}_h {\rm e}^{z y}={\rm e}^{zy} (z-\tanh(y))$.
The functions
\[
\psi_{\mu}^{\widetilde Y}(y):=F\bigl(\sqrt{2(1+\mu)},y\bigr), \qquad \varphi_{\mu}^{\widetilde Y}(y)
:=F\bigl(-\sqrt{2(1+\mu)},y\bigr)
\]
are the increasing/decreasing fundamental solutions to $\L_{\widetilde Y} f=\mu f$ (this follows from \eqref{g_f_intertwining}). With the help of
 \eqref{Prop1:eqn3} we compute \smash{$\wr\big[ \varphi_{\mu}^{\widetilde Y}; \psi_{\mu}^{\widetilde Y}\big]=-
2(1+2\mu)\sqrt{2(1+\mu)}$} and using \eqref{Green_X} we find the Green's function of $\widetilde Y$ in the form
\begin{align}
G^{\widetilde Y}_{\mu}(x,y)={}&-\frac{1}{\sqrt{2(1+\mu)}(1+2\mu)} \nonumber\\
&\times
\begin{cases}
F\bigl(\sqrt{2(1+\mu)},x\bigr) F\bigl(-\sqrt{2(1+\mu)},y\bigr)  &\text{if }  x<y,\\
F\bigl(-\sqrt{2(1+\mu)},x\bigr) F\bigl(\sqrt{2(1+\mu)},y\bigr)  &\text{if }  y<x.
\end{cases}\label{Green_BM_R}
\end{align}
We claim that the transition probability density of $\widetilde Y$ can be written in the following spectral representation form
\begin{equation}\label{p_BM_R_spectral}
 p^{\widetilde Y}_t(x,y)= \frac{{\rm e}^{-\frac{t}{2}}}{2\cosh(x) \cosh(y)} + \frac{1}{2\pi} \int_{\r} {\rm e}^{-(1+\frac{z^2}{2})t}
F(\i z,x) F(-\i z, y) \frac{\d z}{1+z^2}.
\end{equation}
There are three ways how one can obtain this result. The first method is to use general Sturm--Liouville theory (see \cite[Section 2.2]{sousa-yakubovich2022convolution}). The second is to write the transition probability density as the inverse Laplace transform of the Green's function in
\eqref{Green_BM_R}
\[
p^{\widetilde Y}_t(x,y)=\frac{1}{2\pi }
\int_{c+\i \r} G^{\widetilde Y}_{\mu}(x,y) {\rm e}^{\mu t} \d \mu, \qquad c>0,
\]
and transform the contour of integration to a Hankel-type contour, which goes around the interval $(-\infty,-1]$ in the counterclockwise direction (starting at $-\infty$). While transforming this contour of integration we will collect a residue at $\mu=-1/2$, which will give us the first term in \eqref{p_BM_R_spectral}, and the integral over Hankel's contour will give the integral term in \eqref{p_BM_R_spectral}. The third method is probably the simplest (though the least enlightening): we compute the integral in the right-hand side of \eqref{p_BM_R_spectral} and transform the resulting expression into the form
\eqref{p_tildeY1}. More details on this last method are provided in Appendix~\ref{AppendixA}.

The formula \eqref{p_BM_R_spectral} shows that the effect of Darboux transformation on the spectral representation of the transition semigroup of the diffusion process $Y$ is that we have shifted the spectrum by $-1$ and inserted a new eigenvalue at $-1/2$ (with the corresponding eigenfunction~${1/h(x)=1/\cosh(x)}$).

\subsection[Brownian motion on (0,infty), killed at 0]{Brownian motion on $\boldsymbol{(0,\infty)}$, killed at $\boldsymbol{0}$}\label{example2}

Now we take $Y$ to be a Brownian motion on $(0,\infty)$, killed at the first time it hits $0$. The transition probability density is given by
\begin{gather*}
p_t^Y(x,y)=\frac{1}{\sqrt{2\pi t}} \bigl({\rm e}^{-\frac{1}{2t}(y-x)^2}-{\rm e}^{-\frac{1}{2t}(y+x)^2}\bigr)= \frac{2}{\pi} \int_0^{\infty} {\rm e}^{-\frac{z^2}{2} t}
\sin(z x) \sin(z y) \d z, \\ t>0, \qquad x,y >0,
\end{gather*}
see \cite[p.~120]{Borodin2002}. The infinitesimal generator is $\widetilde \L_Y=\frac{1}{2}\partial_y^2$, the killing term $c(x)$ is zero. The point~${l=0}$
is a non-singular boundary for $Y$ and we have a killing boundary condition $f(0+)=0$. We set $\lambda=1/2$ and find a function $h(x)=\sinh(x)$ by solving equation $\L_Y h= \lambda h$ with the boundary condition $h(0+)=0$ and then we verify that
\eqref{lambda_invariant_h} holds with $\lambda=1/2$ (and $l=0$, $r=+\infty$) (or we apply \cite[Theorem 2.7]{Rogers_2021} or \cite[Corollary 2.2]{Mijatovic_2012}). Thus $h$ is $\lambda$-invariant for~$Y$.

We check that Assumption \ref{Assumption_1} holds with ${\mathfrak m}_h=0$ and that Assumption \ref{Assumption_2} is also satisfied, since both boundaries are natural for the process $X$ with the generator given in \eqref{eqn:L_X}. The process $\widetilde Y$ (the Darboux transform of $Y$) is a killed Brownian motion on $(0,\infty)$ with the killing rate $\tilde c(y)=\coth^{2}(y)$. Both boundaries $0$ and $+\infty$ are natural for $\widetilde Y$. The transition probability density of $\widetilde Y$ is
\begin{gather*}
p_t^{\widetilde Y}(x,y)=\frac{{\rm e}^{-t}}{\sqrt{2\pi t}} \bigl({\rm e}^{-\frac{1}{2t}(y-x)^2}+{\rm e}^{-\frac{1}{2t}(y+x)^2}\bigr) + \frac{{\rm e}^{-t/2}}{2\sinh(x) \sinh(y)} \\
\hphantom{p_t^{\widetilde Y}(x,y)=}{}\times\left[
\Phi\left( \frac{y-x-t}{\sqrt{t}} \right) + \Phi\left( \frac{y+x+t}{\sqrt{t}} \right)
-\Phi\left( \frac{y+x-t}{\sqrt{t}} \right)-\Phi\left( \frac{y-x+t}{\sqrt{t}} \right)\right].
\end{gather*}
The fundamental increasing/decreasing solutions to $\L_{\widetilde Y} f = \mu f$ are
\[
 \psi^{\widetilde Y}_{\mu}(x)=\frac{1}{z^2-1} ( z \cosh(zx) - \sinh(zx) \coth(x)), \qquad
 \varphi^{\widetilde Y}_{\mu}(x)={\rm e}^{-zx} (z+\coth(x)),
\]
where we denoted $z=\sqrt{2(\mu+1)}$. The Green's function for the process $\tilde Y$ is
\[
G^{\widetilde Y}_{\mu}(x,y)=\frac{2}{\sqrt{2(\mu+1)}} \times
\begin{cases}
\tilde \psi^+_{\mu}(x) \tilde \varphi^-_{\mu}(y) &\text{if }  x<y,\\
\tilde \varphi^-_{\mu}(x)\tilde \psi^+_{\mu}(y)  &\text{if }  y<x.
\end{cases}
\]
The spectral representation for \smash{$\tilde p^{\widetilde Y}_t(x,y)$} is given by
\begin{equation}\label{tilde_p_spectral2}
p^{\widetilde Y}_t(x,y)=\frac{2}{\pi} \int_0^{\infty} {\rm e}^{-(1+\frac{z^2}{2}) t}
f(z,x) f(z,y) \frac{\d z}{1+z^2},
\end{equation}
where we denoted $f(z,y):=z \cos(zx)-\sin(zx) \coth(x)$. In this case, the Darboux transformation has resulted in shifting the spectrum by $-1$, but no new eigenvalues are created. Intuitively this can be explained by noting that the function $1/h(x)=1/\sinh(x)$ (which would be the candidate eigenfunction with the eigenvalue $-1/2$ for the transition semigroup of $\tilde Y$) is not in~${L_2((0,\infty),\d x)}$.

\subsection[Brownian motion on (0,infty), killed elastically at 0]{Brownian motion on $\boldsymbol{(0,\infty)}$, killed elastically at $\boldsymbol{0}$}\label{example3}

Next, we consider the example when $Y$ is a Brownian motion on $(0,\infty)$ killed elastically at zero,
see \cite[p.~125]{Borodin2002}. In this case, the infinitesimal generator is $\L_Y=\frac{1}{2}\partial_y^2$, the killing term $c(x)$ is zero. The point $l=0$ is a non-singular boundary for $Y$ and we have an elastic boundary condition~${f'\bigl(0^+\bigr)=\gamma f(0+)}$, where $\gamma>0$. The transition probability is
\[
p^Y_t(x,y)=\frac{1}{\sqrt{2\pi t}} \big[ {\rm e}^{-\frac{1}{2t}(y-x)^2}+
 {\rm e}^{-\frac{1}{2t}(y+x)^2} \big]- 2\gamma {\rm e}^{ \gamma(x+y)+\frac{\gamma^2t}{2} } \Phi\left( -\frac{x+y+\gamma t}{\sqrt{t}} \right).
\]
We set $\lambda=1/2$ and solving the equation $\L_Y h=\lambda h$ with the boundary condition $h'\bigl(0^+\bigr)=\gamma f(0+)$ we find a solution
\begin{equation}\label{h3}
h(y)={\rm e}^{y} + \beta {\rm e}^{-y},
\end{equation}
where $\beta:=(1-\gamma)/(1+\gamma)$ and then we check that $h$ is indeed $\lambda$-invariant for $Y$ by verifying that \eqref{lambda_invariant_h} holds.

Assumption \ref{Assumption_1} holds with ${\mathfrak m}_h=0$ and Assumption \ref{Assumption_2} is also satisfied (here we use Proposition \ref{prop_reflecting} for the left boundary). The Darboux transform of $Y$ is a process $\widetilde Y$, which is a Brownian motion on $(0,\infty)$ killed at rate $\tilde c(y)=(h'(y)/h(y))^2$. The left boundary $l=0$ is a non-singular boundary for $\widetilde Y$, and we have a killing boundary condition at this point. In order to simplify the expression for $\tilde c$, we introduce a positive parameter $\alpha$ such that $|\beta|=\exp(-2 \alpha)$ and with the help of this parameter we can express $h$ in the following form: $h(y)=2 {\rm e}^{-\alpha} \cosh(y+\alpha)$ if~${\gamma \in (0,1)}$, $h(y)=\exp(y)$ if $\gamma=1$ and $h(y)=2 {\rm e}^{-\alpha} \sinh(y+\alpha)$ if $\gamma>1$. Thus, the killing rate of~$\widetilde Y$ is given by
\[
\tilde c(y)=
\begin{cases}
\tanh^2(y+\alpha) &\text{if }  \gamma \in (0,1),\\
1  &\text{if }  \gamma=1,\\
\coth^2(y+\alpha)  &\text{if }  \gamma>1.
\end{cases}
\]
The transition probability density of $\widetilde Y$ is
\begin{align}
p^{\widetilde Y}_t(x,y)={}&\frac{{\rm e}^{-t}}{\sqrt{2 \pi t}}\big[ {\rm e}^{-\frac{1}{2t}(y-x)^2}-{\rm e}^{-\frac{1}{2t}(y+x)^2} \big]-\frac{8\gamma \beta}{h(x)h(y)} {\rm e}^{\gamma(x+y)+(\frac{\gamma^2}{2}-1)t}\nonumber\\
&\times\sinh(x) \sinh(y) \Phi\left(-\frac{y+x+\gamma t}{\sqrt{t}}\right)+\frac{2\beta {\rm e}^{-\frac{t}{2}}}{h(x)h(y)} \left[\Phi\left(\frac{y+x-t}{\sqrt{t}}\right)
\right.\nonumber\\
&\left.\phantom{\times}{}+\Phi\left(\frac{y-x+t}{\sqrt{t}}\right)-\Phi\left(\frac{y+x+t}{\sqrt{t}}\right)-\Phi\left(\frac{y-x-t}{\sqrt{t}}\right)\right],\label{p_tilde_Y3}
\end{align}
where $h$ is given by \eqref{h3}. Note that when $\gamma=1$, we obtain the following result: Darboux transform of Brownian motion on $(0,\infty)$ killed elastically at zero with boundary condition~${f'(0+)=f(0+)}$ is the Brownian motion on $(0,\infty)$ killed at rate one and killed at the first time it hits zero. The simplicity of this result suggests that it may have a pathwise explanation, but we were not able to find it.

For the process $\tilde Y$ in this section, we did not derive the spectral representation of the semigroup, as the resulting expressions were rather complicated.

The results obtained in this section lead to the following.

\begin{Corollary}
 Let $W$ be a standard Brownian motion, $\gamma \in (0,\infty) \setminus \{1\}$, $\beta=(1-\gamma)/(1+\gamma)$ and $|\beta|=\exp(-2 \alpha)$. Then for $t>0$ and $x,y>\alpha$
 \begin{equation}\label{killed_BM_identity}
 \e_x\Big[ {\rm e}^{-\int_0^t \kappa(W_s) \d s} {\mathbf 1}_{\{ W_t \le y, \min\limits_{0\le s \le t} W_s>\alpha\}}\Big]=\int_0^{y-\alpha} p_t^{\widetilde Y}(x-\alpha,u)\d u,
 \end{equation}
 where
 \[
\kappa(x)=
\begin{cases}
\tanh^2(x)  &\text{if }  \gamma \in (0,1),\\
\coth^2(x)  &\text{if }  \gamma>1,
\end{cases}
\]
and $p_t^{\widetilde Y}(x,y)$ is given by \eqref{p_tilde_Y3}.
\end{Corollary}
\begin{proof}
Using the representation of $\widetilde Y$ as a Brownian motion on $(0,\infty)$, killed at rate $\tilde c(x)=\kappa(x+\alpha)$ and also killed at the first time it reaches zero, we can write
\[
\int_0^{y-\alpha} p_t^{\widetilde Y}(x-\alpha,u)\d u=\p_{x-\alpha}(\widetilde Y_t \le y-\alpha)=\e_{x-\alpha}\Big[ {\rm e}^{-\int_0^t \kappa(B_s+\alpha) \d s} {\mathbf 1}_{\{ B_t \le y-\alpha, \min\limits_{0\le s \le t} B_s>0\}}\Big],
\]
where $B$ is the standard Brownian motion process. To obtain \eqref{killed_BM_identity}, we denote $W_t=B_t+\alpha$ and use translation invariance of Brownian motion.
\end{proof}

Taking derivatives with respect to $y$ and $\alpha$ allows us to obtain the joint probability density function of the Brownian motion killed at rate $\kappa(x)$ and of its running minimum. This extends the results of Sections \ref{example1} and \ref{example2} by providing information on the running minimum of the process $\widetilde Y$ constructed in those sections.

\begin{Remark}
As discussed in the introduction, the Darboux transformation of second-order linear differential operators is invertible. Starting with an operator $\L$ and a positive formal eigenfunction $h$ such that $\L h= \lambda h$, we obtain Darboux transformed operator $\widetilde \L$, for which $1/h$ is also a formal eigenfunction. Applying Darboux transformation to $\tilde \L$ with the seed function~$1/h$, we recover the original operator $\L$. However, this inversion does not work on the level of transition semigroups of diffusion processes. Consider the process $Y$ from this section and the function $h$ given by \eqref{h3}. The function $f=1/h$ is positive and satisfies the equation $\L_{\widetilde Y} f=\lambda f $ (with $\lambda=1/2$), but it is not a $\lambda$-invariant function for \smash{$\widetilde Y$}, since it does not satisfy the boundary condition at $0$ (recall that the process $\widetilde Y$ has a killing boundary condition at zero).
\end{Remark}

\subsection[Brownian motion on (0,1), killed at 0 or 1]{Brownian motion on $\boldsymbol{(0,1)}$, killed at $\boldsymbol{0}$ or $\boldsymbol{1}$}\label{example4}

Let $Y$ be a Brownian motion on $(0,1)$ killed at both boundaries, see \cite[p.\ 122]{Borodin2002}. The transition probability density is
\[
p^Y_t(x,y)=2\sum\limits_{n\ge 1} {\rm e}^{-\frac{1}{2} \pi^2 n^2 t}
\sin(\pi n x) \sin(\pi n y).
\]
The above formula can be seen as a spectral representation of the transition density of $Y$ in~$L_2((0,1),\d x)$. Using orthogonality of $\{\sin(n \pi y)\}_{n\ge 1}$ on the interval $(0,1)$, it is easy to check that~${h(y)=\sin(\pi y)}$ is $\lambda$-invariant function for $Y$ with $\lambda=-\pi^2/2$. We check that Assumption \ref{Assumption_1} holds with ${\mathfrak m}_h=0$ and that Assumption~\ref{Assumption_2} is also satisfied, since both boundaries are natural for the process $X$ with the generator given in \eqref{eqn:L_X}. The Darboux transformed process~$\widetilde Y$ is a~Brownian motion on $(0,1)$, killed at rate
$\tilde c(x)= \pi^2 \cot^{2} (\pi x)$. Both boundaries are natural for~$\widetilde Y$. The transition probability density is obtained using \eqref{equation_tilde_p} and has the form
\begin{equation}\label{BM_4_tilde_p}
p^{\widetilde Y}_t(x,y)=2\sum\limits_{n\ge 2} \frac{{\rm e}^{-\frac{1}{2} \pi^2 ( n^2-2) t} }{n^2-1} f_n(x) f_n(y),
\end{equation}
where we denoted
\[
f_n(x):=\frac{1}{\pi} {\mathcal D}_h \sin(n \pi x)=n \cos(\pi n x)- \sin(\pi n x) \cot(\pi x).
\]
The functions $\{f_n\}_{n\ge 0}$ are orthogonal in $L_2((0,1),\d x)$
\[
\int_0^1 f_n(x) f_m(x) \d x=\frac{1}{2} \bigl(n^2-1\bigr) \delta_{n,m}, \qquad n,m\ge 2.
\]
This result follows by applying \eqref{Prop1:eqn2} with $f=\sin(n \pi x)$ and $g=\sin(m \pi x)$. Thus formula \eqref{BM_4_tilde_p} gives a spectral representation of the transition probability density of $\widetilde Y$ in $L_2((0,1),\d x)$. In this case, the Darboux transformation removes the first eigenvalue from the spectrum of $\L_Y$ and shifts all eigenvalues by $-2 \lambda$.

Now we can repeat this construction. We note that $f_2(x)=-2\sin^2(\pi x)$, thus $h_1(y)=-f_2(y)=2\sin^2(\pi y)$ is a positive $\lambda$-invariant function for $\widetilde Y$, with $\lambda=-\pi^2$. Applying Darboux transformation to $\widetilde Y$ with this $\lambda$-invariant function $\tilde h$, we obtain a new diffusion process, which we denote by $Y^{(2)}$. This process is a Brownian motion on $(0,1)$ killed at rate
$c^{(2)}(x)=3\pi^2 \cot^2(\pi x)$ and having transition probability density
\begin{equation}\label{tildep_2}
p^{(2)}_t(x,y)=2\sum\limits_{n\ge 3} \frac{{\rm e}^{-\frac{1}{2} \pi^2 (n^2-6) t} }{(n^2-1)(n^2-4)} f^{(2)}_n(x) f^{(2)}_n(y),
\end{equation}
where
\[
f^{(2)}_n(x):=\frac{1}{\pi} {\mathcal D}_{h_1} f_n(x)=-(2+n^2) \sin(\pi n x) + 3\sin(\pi n x) \csc^2(\pi x) - 3 n \cos(\pi n x)\cot(\pi x).
\]
The formula \eqref{tildep_2} again gives us a spectral representation of the transition probability density of the process $Y^{(2)}$ in $L_2((0,1), \d x)$.

This process can be repeated indefinitely. The first eigenfunction of the transition semigroup must have a constant sign on $(0,1)$, thus we can take $\lambda$-invariant function $h$ to equal to this first eigenfunction or its negative. After performing Darboux transform $m$ times, we will obtain a~process $Y^{(m)}$, which is a Brownian motion on $(0,1)$, killed at rate $c^{(2)}(x)=\frac{1}{2} m(m+1)\pi^2 \cot^2(\pi x)$. The transition probability can be written in the following form:
\begin{equation}\label{tildep_m}
p^{(m)}_t(x,y)=2\sum\limits_{n\ge m+1} \frac{{\rm e}^{-\frac{1}{2} \pi^2 (n^2-m(m+1)) t}}
{\prod\limits_{j=1}^m \bigl(n^2-j^2\bigr)} f^{(m)}_n(x) f^{(m)}_n(y),
\end{equation}
where the functions $f^{(m)}_n$ can be given explicitly in terms of Gegenbauer polynomials and also in terms of certain Wronskian determinants. The proof of \eqref{tildep_m} will appear in the forthcoming paper \cite{Yuan_2024}.

\subsection[Brownian motion on R killed at rate y\^2/2]{Brownian motion on $\boldsymbol{\r}$ killed at rate $\boldsymbol{y^2/2}$}\label{example5}

Let $Y$ be a Brownian motion on $\r$ killed at rate $c(y)=y^2/2$. It is known that its transition probability density function is given by
Mehler's kernel
\begin{align*}
p^Y_t(x,y)&=\frac{1}{\sqrt{2\pi \sinh(t)}}
\exp\left(- \frac{1}{2}\coth(t) \bigl(x^2+y^2\bigr) + \frac{xy}{\sinh(t)} \right)\\
&=
\frac{1}{\sqrt{\pi}} {\rm e}^{-\frac{1}{2} (x^2+y^2 )}\sum\limits_{n\ge 0} \frac{{\rm e}^{- (n+\frac{1}{2}) t}}{2^n n!} H_n(x) H_n(y).
\end{align*}
Here $H_n(x)$ denote Hermite polynomials.
The process $Y$ is closely related to the Ornstein--Uhlenbeck process $Z$, which has infinitesimal generator $\L_Z=\frac{1}{2} \partial_z^2 - z \partial_z$.
It is known (and easy to check) that $Y$ is Doob's $h$-transform of $Z$
\[
\p_z(Y_t \in A)={\rm e}^{-\frac{t}{2}-\frac{1}{2} z^2} \e_z\big[ {\rm e}^{\frac{1}{2} Z_t^2} {\mathbf 1}_{\{Z_t \in A\}}\big].
\]

Next, we will discuss how to find a $\lambda$-invariant function for the process $Y$. It is known (see~\cite[Example 2.3]{Gomez_2020}) that the functions $f_n(y)={\rm e}^{-y^2/2}H_n(y)$ and $ g_n(y)=\i^n {\rm e}^{y^2/2} H_n(\i y)$ are formal eigenfunctions of the operator $\L=\frac{1}{2} \bigl(\partial_y^2 - y^2\bigr)$. Hermite polynomials $H_n(y)$ have real zeros if~$n\ge 1$ and the functions $H_n(\i y)$ have a zero at $y=0$ when $n$ is odd, so these are not good candidates for
a positive $\lambda$-invariant function and we should take $h$ as a multiple of~$f_0$ or~$g_{2n}$. Taking $h(y)=\exp\bigl(\pm y^2/2\bigr)$ does not lead to a new process, as $\tilde c(y)$ given by \eqref{new_tilde_c} is equal to~$c(y)$ up to a constant. Thus, the first non-trivial example can be obtained when we take~$h$ as a~multiple of $g_2(y)$. We set $h(y)={\rm e}^{y^2/2} \bigl(2y^2+1\bigr)$ and check (by computing the integral in~\eqref{lambda_invariant_h} symbolically
or by applying \cite[Theorem 2.7]{Rogers_2021} or \cite[Corollary 2.2]{Mijatovic_2012}) that $h$ is $\lambda$-invariant for $Y$ with $\lambda=5/2$. Now we can construct the Darboux transformed process $\widetilde Y$, which is
 a Brownian motion on $\r$ killed at rate
\[
\tilde c(y)=\frac{y^2}{2}+\frac{8 y^2 \bigl(2y^2 + 3\bigr)}{\bigl(2y^2+1\bigr)^2}.
\]
The transition probability density of $\widetilde Y$ (computed via \eqref{equation_tilde_p}) has a surprisingly simple form
\begin{equation}\label{eq:tilde_p_OU}
p^{\widetilde Y}_t(x,y)
=p^Y_t(x,y) \times {\rm e}^{-4t}
\left[ 1 + \frac{4 \sinh(t) \bigl({\rm e}^{t}-2xy\bigr)}{\bigl(2x^2+1\bigr)\bigl(2y^2+1\bigr)}\right].
\end{equation}

Next, we present the spectral expansion of the transition probability density. The operator~${\L_Y=\frac{1}{2} \bigl(\partial_y^2 - y^2\bigr)}$ has a complete set of orthogonal eigenfunctions $f_n(y)={\rm e}^{-y^2/2}H_n(y)$ in~$L_2(\r,\d x)$. We know that $1/h(y)$ is a formal eigenfunction of \smash{$\widetilde Y$}. We can expect that \smash{$\tilde f_n(y) ={\mathcal D}_h f_n(y)$}
to be the eigenfunctions of $\L_{\widetilde Y}$ in $L_2(\r, \d x)$. We compute
\[
\tilde f_n(y) ={\mathcal D}_h f_n(y)=f_n'(y)-\frac{h'(y)}{h(y)}f_n(y)=\frac{Q_n(y)}{h(y)},
\]
where $Q_n(y)$ is the following polynomial
\[
Q_n(y)=\wr[h,f_n]=2n\bigl(2y^2+1\bigr) H_{n-1}(y)-\bigl(4y^3+6y\bigr) H_n(y), \qquad n\ge 0.
\]
Note that $Q_n$ are polynomials of degree $n+3$.
\begin{Proposition}
The set of polynomials $\{1\} \cup \{Q_n\}_{n \ge 0}$ is a complete orthogonal set in the space $L_2(\r, \nu(\d y))$, where
\[
\nu(\d y):=h^{-2}(y) \d y=\frac{{\rm e}^{-y^2}}{\bigl(2y^2+1\bigr)^2} \d y.
\]
The transition probability density of $\widetilde Y$ has spectral representation
\begin{equation}\label{OU_spectral}
p^{\widetilde Y}_t(x,y)=
\frac{1}{\sqrt{\pi}} \frac{{\rm e}^{-\frac{1}{2}(x^2+y^2)}}
{\bigl(2x^2+1\bigr)\bigl(2y^2+1\bigr)} \times
\left[ 2 {\rm e}^{-\frac{5}{2} t} +\sum\limits_{n\ge 0} {\rm e}^{-(n+11/2) t} \frac{Q_n(x) Q_n(y)}
{2^{n+1} n! (n+3)} \right].
\end{equation}
\end{Proposition}
\begin{proof}
Let $S=\{0,3,4,5,6,\dots\}$ be the set of non-negative integers that is missing $\{1,2\}$ and let us define polynomials \smash{$H^{(1)}_n(y)$} via
\[
H^{(1)}_n(y)=\wr[H_1,H_2,H_n](y)=16 H_n(y)-16 x H_n'(y)+4 \bigl(2x^2+1\bigr) H_n''(y).
\]
Note that \smash{$H^{(1)}_n \neq 0$} if and only if $n\in S$.
The polynomials \smash{$\big\{H^{(1)}_n\big\}_{n\in S}$} are the exceptional Hermite polynomials corresponding to the partition $2=1+1$, see \cite[Definition 5.1]{Ullate_2014}. It is known (see~\cite[Propositions 5.7 and 5.8]{Ullate_2014}) that \smash{$\big\{H^{(1)}_n\big\}_{n \in S}$} form a complete orthogonal set in $L_2(R,\nu(\d x))$ with the squared norm
\begin{equation}\label{H1_n_squared_norm}
\int_{\r} H^{(1)}_n(y)^2 \nu(\d y)=\sqrt{\pi} 2^{n+6} n! (n-1)(n-2).
\end{equation}
Using \cite[Theorem 3.1]{Ullate_2018}, we can write
\begin{equation}\label{H1_n_as_Qn}
H^{(1)}_n(y)=-16 (n+1)(n+2) Q_n(y), \qquad n\ge 0,
\end{equation}
and clearly \smash{$H^{(1)}_n(y)=16$}. Thus the polynomials \smash{$\big\{H^{(1)}_n\big\}_{n\in S}$} and $\{1\} \cup \{Q_n\}_{n \ge 0}$ are identical, up to a multiplicative constant. This implies that the functions $\{1/h\} \cup \{Q_n/h\}_{n \ge 0}$ form a complete orthogonal set of eigenfunctions of $\L_{\widetilde Y}$ in $\L_2(\r,\d x)$ with eigenvalues $-5/2$ and $-(n+11/2)$, $n\ge 0$. The $L_2$-norm of these functions can be found via \eqref{H1_n_squared_norm} and \eqref{H1_n_as_Qn} (though one could find the norms of $Q_n/h$ more directly via Lemma \ref{lemma_Dh}
\[
\int_{\r} \frac{1}{h(y)^2} \d y=\nu(\r)=\frac{\sqrt{\pi}}{2}, \qquad \int_{\r} \frac{Q_n(y)^2}{h(y)^2} \d y=
\int_{\r} Q_n(y)^2 \nu(\d y)=\sqrt{\pi} 2^{n+1} n! (n+3).
\]
One can now write down the transition probability density of $\widetilde Y$ in the form \eqref{OU_spectral}.
\end{proof}

In conclusion, we would like to discuss the connection of our results with the existing literature on Darboux transformation and propagators of
for the one-dimensional Schr\"odinger equation. Our formula \eqref{equation_tilde_p} has an analogue in \cite[Theorem~2]{Pupasov_2007}.
The example~in Section~\ref{example1} should be compared with the results of \cite{Pupasov_2005} and the example in Section~\ref{example4} with \cite[Example~5.1]{Pupasov_2007}.
Our example in this section is closely related to the propagators of rational extensions of harmonic oscillator considered in \cite{Pupasov_2015}. Moreover, the results in \cite{Pupasov_2015} imply that if the potential
\[
u(y)=y^2 + \frac{a(y)}{b(y)}
\]
is a rational extension of harmonic oscillator (as defined in \cite{Ullate_2014}) then the process $\widetilde Y$, which is the Brownian motion on $\r$ killed at rate $\tilde c(y)=u(y)/2$, will have transition probability density of the form similar to \eqref{eq:tilde_p_OU}
\[
p_t^{\widetilde Y}(x,y)=p_t^Y(x,y) \times {\rm e}^{-c t} \frac{Q\bigl(x,y,{\rm e}^{t}\bigr)}{R(x,y)},
\]
where $c$ is a constant and $Q(x,y,z)$ and $R(x,y)$ are certain polynomials that can be computed explicitly. Thus, there exists a hierarchy of diffusion processes on $\r$, which could be called rational extensions of Ornstein--Uhlenbeck process, for which the transition probability density can be computed in a fairly simple form. We plan to investigate this family of diffusion processes in future work.

\appendix

\section[Proof of identities (5.7) and (5.8)]{Proof of identities (\ref{p_BM_R_spectral}) and (\ref{tilde_p_spectral2})}\label{AppendixA}

We begin with the following result
\begin{equation}
 \label{Phi_integral_iz}
 \frac{{\rm e}^{w}}{2\pi } \int_{\r} \frac{{\rm e}^{-t\frac{z^2}{2}+\i zw }}{\i z+1} \d z=
 {\rm e}^{\frac{t}{2}} \Phi\left(\frac{w-t}{\sqrt{t}}\right), \qquad t>0, \quad w\in \r,
\end{equation}
which can be established by taking the derivative in $w$ of both sides and computing the integral in the left-hand side. From
\eqref{Phi_integral_iz}, we obtain the following pair of Fourier transform identities:
\begin{gather}
 \frac{1}{2\pi} \int_{\r} \frac{{\rm e}^{- t \frac{z^2}2 + \i z w }}{1+z^2} \d z=
 \frac{1}{2} \left[ {\rm e}^{\frac{t}{2} + w} \Phi\left( - \frac{w+t}{\sqrt{t}}\right) +
 {\rm e}^{\frac{t}{2} - w} \Phi\left( \frac{w-t}{\sqrt{t}}\right)\right],\label{Fourier1}\\
 \frac{\i}{2\pi} \int_{\r} \frac{{\rm e}^{- t \frac{z^2}2 + \i z w }}{1+z^2} z \d z=
 \frac{1}{2} \left[ {\rm e}^{\frac{t}{2} + w} \Phi\left( - \frac{w+t}{\sqrt{t}}\right) -
 {\rm e}^{\frac{t}{2} - w} \Phi\left( \frac{w-t}{\sqrt{t}}\right)\right]. \label{Fourier2}
\end{gather}
We write
\[
F(\i z ,x) F(-\i z,y)={\rm e}^{\i z (x-y) }
\big[ \bigl(1+z^2\bigr) + \i z (\tanh(x)-\tanh(y))+(\tanh(x)\tanh(y)-1)\big]
\]
and now we can compute the integral in \eqref{p_BM_R_spectral}
\begin{gather*}
\frac{1}{2\pi} \int_{\r} {\rm e}^{- (1+\frac{z^2}{2})t}
F(\i z,x) F(-\i z, y) \frac{\d z}{1+z^2}\\
\qquad=\frac{1}{\sqrt{2\pi t}} {\rm e}^{-t-\frac{1}{2t}(x-y)^2} +\frac{1}{2} (\tanh(x)-\tanh(y))\\
\phantom{\qquad=}{}\times \left[ {\rm e}^{-\frac{t}{2} + (x-y)} \Phi\left( \frac{y-x-t}{\sqrt{t}}\right) -
 {\rm e}^{-\frac{t}{2} - (x-y)} \Phi\left( \frac{x-y-t}{\sqrt{t}}\right)\right]\\
\phantom{\qquad=\times}{}+\frac{1}{2} \bigl(\tanh(x) \tanh(y) -1)\\
\phantom{\qquad=}{}\times
\left[ {\rm e}^{-\frac{t}{2} + (x-y)} \Phi\left( \frac{y-x-t}{\sqrt{t}}\right) +
 {\rm e}^{-\frac{t}{2} - (x-y)} \Phi\left( \frac{x-y-t}{\sqrt{t}}\right)\right].
\end{gather*}
Plugging the above expression into the right-hand side of \eqref{p_BM_R_spectral} and simplifying the result will give us \smash{$p^{\widetilde Y}_t(x,y)$}.

The proof of \eqref{tilde_p_spectral2} is very similar. We write the integral in the right-hand side of \eqref{tilde_p_spectral2} as
a~sum of five integrals
\begin{gather*}
 \int_0^{\infty} {\rm e}^{-(1+\frac{z^2}{2}) t}
f(z,x) f(z,y) \frac{\d z}{1+z^2}\\
\qquad= \int_0^{\infty} {\rm e}^{-(1+\frac{z^2}{2}) t}
\cos(zx)\cos(zy) \d z
- \coth(x) \int_0^{\infty} {\rm e}^{-(1+\frac{z^2}{2}) t} \frac{\cos(zy) \sin(zx)}{1+z^2}z \d z\\
\phantom{\qquad= }{}- \coth(y) \int_0^{\infty} {\rm e}^{-(1+\frac{z^2}{2}) t} \frac{\cos(zx) \sin(zy)}{1+z^2} z \d z
\\
\phantom{\qquad= }{}+ \coth(x) \coth(y)
 \int_0^{\infty} {\rm e}^{-(1+\frac{z^2}{2}) t} \frac{\sin(zx) \sin(zy)}{1+z^2} \d z
\\
\phantom{\qquad= }{} -
 \int_0^{\infty} {\rm e}^{-(1+\frac{z^2}{2}) t} \frac{\cos(zx) \cos(zy)}{1+z^2} \d z.
\end{gather*}
Each of these integrals can be computed by using product-to-sum identities for the trigonometric functions and applying Fourier transform
identities \eqref{Fourier1} and \eqref{Fourier2}. The resulting expression, when plugged into the right-hand side of \eqref{tilde_p_spectral2}, after some long and tedious simplifications, will give us \smash{$p^{\widetilde Y}_t(x,y)$}.

\section[Proof of (ii)]{Proof of (ii)}\label{AppendixB}

We present the proof in the case when $l$ is a non-singular boundary for the process $Y$ from Section~\ref{section_Darboux_diffusion}. The proof for the case when $r$ is non-singular is identical. We recall that $\L_Y = \frac{1}{2} \partial_y^2 - c(y)$ (where $c$ is a positive continuous function on $(l,r)$) and that the boundary $l$ is non-singular for $Y$ if and only if $l$ is finite and
\begin{equation}\label{c_L1}
\int_l^z c(y) \d y<\infty,
\end{equation}
for some $z\in (l,r)$. Let $\lambda$ be a real number and $h$ be a $\lambda$-invariant function for $Y$.

First, let us establish the following result: any positive solution to $\L_Y f=\lambda f$ must be bounded on $(l,z)$. This result is undoubtedly known in Sturm--Liouville theory, but we were unable to find a reference, so we will include a short proof here. We denote $g(y)=f'(y)/f(y)$ and check that $g$ satisfies the Riccati equation
$
g'(y)=2 (c(y)+\lambda)-g(y)^2$, $ y\in (l,r)$,
which implies (by integrating over $y\in (w,z)$)
\[
g(z)-g(w) \le 2 \lambda (z-w) + 2 \int_w^z c(y) \d y, \qquad l<w<z<r.
\]
Therefore, $-g(w)$ is bounded from above on the interval $(l,z)$ (due to \eqref{c_L1}), and since
\[
f(y)=f(z) \exp\left( -\int_y^z g(w) \d w\right), \qquad l<y<z,
\]
we conclude that $f$ is also bounded on $(l,z)$.

Now we are ready to prove that $h$ satisfies the appropriate boundary condition at $l$. We denote by $\psi_{\mu}$ and $\varphi_{\mu}$ the fundamental increasing/decreasing solutions to $\L_Y f = \mu f$. We will show that~$h$ satisfies the same boundary condition at $l$ as $\psi_{\mu}$.
The starting point is the identity~\eqref{h_lambda_invariant2}, which, coupled with the representation \eqref{Green_X} for the Green's function of the process $Y$, gives~us
\[
h(x)=(\mu-\lambda) \left[ \varphi_{\mu}(x) \int_l^x \psi_{\mu}(y) h(y) \d y+ \psi_{\mu}(x) \int_x^r \varphi_{\mu}(y) h(y) \d y \right].
\]
The above holds for $\mu > \max(0,\lambda)$ and $x\in (l,r)$. As we established above, both functions $\varphi_{\mu}$ and $h$ are bounded (thus, integrable) on $(l,z)$, thus we can rewrite the above identity in the form
\begin{align*}
h(x)={}&(\mu-\lambda) \psi_{\mu}(x) \int_l^r \varphi_{\mu}(y) h(y) \d y+ (\mu-\lambda) \varphi_{\mu}(x) \int_l^x \bigl(\psi_{\mu}(y)- \psi_{\mu}(x)) h(y) \d x \\
&- (\mu-\lambda)
\psi_{\mu}(x) \int_l^x \bigl(\varphi_{\mu}(y)- \varphi_{\mu}(x)) h(y) \d x.
\end{align*}
Since $\psi_{\mu}$, $\varphi_{\mu}$ and $h$ are bounded on $(l,z)$, the above identity implies that as $x\to l+$ we have
\begin{gather}
h(x)=C \psi_{\mu}(x) + O(x-l), \label{h_psi_mu}\\
h'(x)=C \psi'_{\mu}(x) + \varphi'_{\mu}(x) O(x-l)+
\psi'_{\mu}(x) O(x-l),\label{h_psi_mu2}
\end{gather}
where we denoted
\[
C=C(\mu,\lambda):=(\mu-\lambda) \int_l^r \varphi_{\mu}(y) h(y) \d y.
\]
If we have a killing boundary condition for $Y$ at point $l$, that is $f(l+)=0$, then \eqref{h_psi_mu} implies~${h(l+)=\psi_{\mu}(l+)=0}$. If we have a reflecting boundary condition at $l$, then $\psi_{\mu}'(l+)=0$ and~\eqref{h_psi_mu2} implies $h'(l+)=0$. Here we also need the fact that $\psi_{\mu}'(x)$ and $\varphi_{\mu}'(x)$ are bounded near point $l$ (the first one is bounded near $l$ since $\psi_{\mu}(x)$ is convex and increasing and the second one is bounded because the Wronskian
$\wr[\varphi_{\lambda}; \psi_{\lambda}]$ is a non-zero constant).
In the case when we have an elastic boundary condition for $Y$ at point $l$, that is $f(l+)=\gamma f'(l+)$, we can obtain~${h(l+)=\gamma h'(l+)}$ using the same argument.

\subsection*{Acknowledgements} The research was supported by the Natural Sciences and Engineering Research Council of
Canada. The authors would like to thank Mateusz Kwa\'snicki for stimulating discussions and for helping with the proof of Theorem~\ref{thm:main}. We are also grateful to anonymous referees for carefully reading the paper and for providing very helpful comments and suggestions.

\pdfbookmark[1]{References}{ref}
\LastPageEnding

\end{document}